\documentclass{article}
\RequirePackage[left=1in,right=1in,top=1in,bottom=1in]{geometry}
                                                          
\usepackage[usenames,dvipsnames]{color}
\usepackage{amsthm}
\usepackage{amssymb}
\usepackage{color}
\usepackage{amsmath}
\usepackage{mathtools}
\usepackage{cite}
\usepackage[normalem]{ulem}
\usepackage{enumitem}

\DeclareMathOperator*{\argmin}{argmin}

\usepackage{graphics} 
\usepackage{epsfig} 
\usepackage{float}
\usepackage{caption}
\usepackage{subcaption}

\usepackage{algorithm}
\usepackage{algorithmicx}
\usepackage{algpseudocode}
\usepackage{setspace}

\title{\LARGE {\bf{Output Observability of Systems Over Finite Alphabets with Linear Internal Dynamics}}}

\author{Donglei Fan and Danielle C. Tarraf
\thanks{The authors were with the Department of Electrical and Computer Engineering at Johns Hopkins University, Baltimore, MD 21218, USA when this research was conducted ({\tt\small dongleifan4@gmail.com, dtarraf@alum.mit.edu}). D.C. Tarraf is a Visiting Scholar at the MIT Institute for Data, Systems and Society.} %
\thanks{This research was supported by AFOSR grant FA9550-16-1-0132, NSF CAREER grant 0954601 and AFOSR Young Investigator grant FA9550-11-1-0118 while the authors were at Johns Hopkins University.}
}
\date{July 10, 2016}

\theoremstyle{remark}     
\theoremstyle{remark}     \newtheorem{lemma}{Lemma}
\theoremstyle{remark}     \newtheorem{mydef}{Definition}
\theoremstyle{remark}     
\theoremstyle{remark}     \newtheorem{theorem}{Theorem}
\theoremstyle{remark}     
\theoremstyle{remark}     \newtheorem*{rmk}{Remark}
\theoremstyle{remark}     \newtheorem{eg}{Example}
\theoremstyle{remark}     
\theoremstyle{remark}     \newtheorem{cor}{Corollary}

\renewcommand{\qedsymbol}{$\blacksquare$}

\begin{document}

\maketitle
\thispagestyle{empty}

\begin{abstract}
We consider a class of systems over finite alphabets with linear internal dynamics, finite-valued control inputs and finitely quantized outputs. 
We motivate the need for a new notion of observability
and propose three new notions of {\it output} observability, thereby shifting our attention to the problem of state estimation for output prediction.
We derive necessary and sufficient conditions for a system to be output observable,
algorithmic procedures to verify these conditions, 
and a construction of {\it finite memory} output observers when certain conditions are met.
We conclude with simple illustrative examples.  

\end{abstract}

\section{Introduction}

We study observability for a class of systems over finite alphabets \cite{DT}, namely systems with linear 
internal dynamics and finitely quantized outputs. 
The plant thus consists of a discrete-time linear time-invariant (LTI) system with  
a finite control input set and a saturating output quantizer that restricts the values of the sensor output signals to fixed, finite sets. 
In our past research on systems over finite alphabets, we introduced a notion of finite state `$\rho/\mu$ approximation' \cite{DT12}.
The idea there is to construct a sequence of deterministic finite state machines that satisfy a set of properties,
thereby constituting approximate models that can be used as the basis for certified-by-design control synthesis \cite{DT11}:
Specifically, a full state feedback control law is first designed for the approximate model to achieve a suitably defined auxiliary performance objective.
This control law is then used, 
together with a copy of the approximate model serving as a {\it finite memory observer} of the plant,
to certifiably close the loop around the system.
This sequence of developments brings to the forefront the 
problem of state estimation for systems over finite alphabets.
The results reported in this manuscript constitute a step towards addressing that problem 
for the specific class of systems considered, namely those with LTI internal dynamics.

Recall that observability generally refers to the ability to determine the initial state of a system from a single observation of its input and output over some finite time interval. In particular, an LTI system is observable if and only if different initial states produce different outputs under zero input \cite{Hespanha}.
Similarly, a nonlinear system is locally observable at $x_o$ if there exists some neighborhood of $x_o$ in which different initial states produce different outputs from that of $x_o$ under every admissible input \cite{hermann1977nonlinear}.

The problem of observability of hybrid systems, 
including switched linear systems \cite{Balluchi, DL, Xie} and quantized-output systems \cite{DD, JS, Raisch}, has been studied in recent years. 
The results in \cite{yuksel2007communication} and \cite{yuksel2006minimum} are also closely related to the problem of state estimation based 
on quantized sensor output information.
However at this time, we are not aware of work on observability of discrete-time systems that involve both switching and output quantization, apart from our work in \cite{DF14} in which we presented a subset of the results in the present manuscript.

The problem of observer design, particularly in a discrete-state setting, has also been studied recently.
For instance, \cite{Murray06} proposed discrete state estimators to estimate the discrete variables in hybrid systems where the continuous variables are available for measurement, while \cite{Topcu15} and \cite{Ozay14} proposed finite-state and locally affine estimators, respectively, for systems whose control specifications are expressed in temporal logic.
 
As we shall see in what follows, the traditional concept of observability does not generalize well to the class of systems of interest. 
Therefore, inspired by our work on $\rho/\mu$ approximations \cite{DT} \cite{DT13}, we propose to shift our attention from state estimation to state estimation for output prediction, emphasizing in particular deterministic finite state machine (DFM) observers. 
The main contributions of this manuscript are as follows:
\begin{enumerate}
\item We motivate the need for a new notion of observability for systems over finite alphabets.
\item Shifting our emphasis from state estimation to state estimation for the purpose of output prediction, we propose three new associated notions:
Finite memory output observability, weak output observability  and asymptotic output observability.
\item We characterize necessary and sufficient conditions for output observability in terms of the parameters of the system for a class of systems over finite alphabets with linear internal dynamics.
\item We propose an algorithm for verifying some of the the sufficient conditions.
\item We propose a constructive procedure for generating finite memory output observers when certain sufficient conditions are met.
\end{enumerate}

\emph{Organization:}
We introduce the class of systems of interest in Section \ref{Sec:Systems}.
We motivate the need for a new notion of observability and propose three new notions of output observability in Section \ref{sec:defn-obs}. 
We investigate these three notions, derive a set of necessary and sufficient conditions, an algorithmic procedure for verifying some of these conditions, and a finite memory observer construction in Sections \ref{sec:FMO} and \ref{sec:WO}.
We present illustrative examples in Section \ref{sec:Eg}
and conclude with directions for future work in Section \ref{Sec:Conclusions}.

\emph{Notation:}
We use $\mathbb{N}$ to denote the non-negative integers, $\mathbb{Z}_+$ to denote the positive integers, $\mathbb{R}_{\ge 0}$ to denote the non-negative reals, and $\mathbb{R}_+$ to denote the positive reals.
We use $ \mathcal{A}^{\mathbb{N}}$ to denote the collection of infinite sequences over set $\mathcal{A}$, that is  $ \mathcal{A}^{\mathbb{N}} = \{ f : \mathbb{N} \to \mathcal{A} \}$.
For $\bf{a} \in \mathcal{A}^{\mathbb{N}}$, we use $a_t$ to denote its $t^{th}$ component. We use $\{a_t\}_{t \in \mathcal{I}}$ to denote the subsequence over index set $\mathcal{I} \subset \mathbb{N}$.
Given a function $f: \mathcal{X} \to \mathcal{Y}$, we use $f^{-1}(y)$ to denote the inverse image of $y \in \mathcal{Y}$ under $f$. For two positive integers $a$ and $b$, we use $a \hspace{-4pt}\mod b$ to denote the remainder of the division of $a$ by $b$.

For $v \in \mathbb{R}^n$, we use $\|v\|$ to denote the Euclidean norm, $\|v\|_1$ to denote the $1$-norm, and $\|v\|_{\infty}$ to denote the infinity norm. 
We say $v$ is bounded if there exists a $b \in \mathbb{R}_{\ge 0}$ such that $\|v\| \le b$.
For a square matrix $A$, we use $\|A\|$ to denote the induced $2$-norm, $\|A\|_1$ to denote the induced $1$-norm, and $\|A\|_{\infty}$ to denote the induced infinity norm. 
We use $\rho(A)$ to denote the spectral radius of $A$, and we say that $A$ is Schur-stable if $\rho(A) < 1$. We say $v$ is a generalized eigenvector of matrix $A$  with corresponding eigenvalue $\lambda$ if $(A - \lambda I)^l v = 0$ but  $(A - \lambda I)^{l-1} v \neq 0$ for some integer $l \ge 1$. 
We use ${\bf 0}$ to represent the zero matrix of appropriate dimensions.
For  $w \in \mathbb{C}^p$, we use $[w]_i$ to denote its $i^{th}$ component and $Re(w)$ to denote the (vector) real part of $w$.

\color{black}
We use $B_r(v)$ and $\overline{B_r(v)}$ to denote the open and closed balls, respectively, centered at $v$ with radius $r$.
For sets $\mathcal{A}$, $\mathcal{B}$ in $\mathbb{R}^n$, we use $|\mathcal{A}|$ denote the cardinality of $\mathcal{A}$ and $d(\mathcal{A}, \mathcal{B}) = \inf \{\|\alpha - \beta\|: \alpha \in \mathcal{A}, \beta \in \mathcal{B}\}$ to denote the distance between sets $\mathcal{A}$ and $\mathcal{B}$. 
Given a finite ordered set $\mathcal{S} = \{s_1, s_2, \dots, s_l\}$ where $l \in \mathbb{Z}_+$, we use $\{s_j\}_{j=1}^l$ to denote  $\mathcal{S}$.

\section{Systems of Interest}
\label{Sec:Systems}

A \emph{system over finite alphabets} is understood to be a set of pairs of signals,
\begin{equation}
\label{eq:Sofa-IO}
P\subset \mathcal{U}^{\mathbb{N}} \times \mathcal{Y}^{\mathbb{N}}, 
\end{equation}
with $|\mathcal{U}| < \infty$ and 
$|\mathcal{Y}| < \infty$.
Essentially, $P$ is a discrete-time system whose input sequences and corresponding feasible output sequences are defined over {\it finite} input and output sets, $\mathcal{U}$ and $\mathcal{Y}$, respectively.
While this definition is quite broad, and we indeed studied these systems in a general setting in \cite{DT},
in this manuscript we are interested in instances where $P$ has underlying dynamics evolving in a continuous state-space described by
\begin{subequations}
\label{eq:Sofa}
\begin{align}
x_{t+1}  &=  f_P(x_t, u_t), \\
y_t  &=  g_P(x_t,   u_t), 
\end{align}
\end{subequations}
where $t \in \mathbb{N}$ is the time index, $x_t \in \mathbb{R}^{n}$ is the state, 
$u_t \in \mathcal{U} \subset \mathbb{R}^m$ is the input ($|\mathcal{U}| < \infty$), $y_t\in \mathcal{Y}  \subset \mathbb{R}^p$ is the output ($|\mathcal{Y}| < \infty$), $f_P: \mathbb{R}^{n} \times \mathcal{U} \to \mathbb{R}^{n} $ is the state transition function and $g_P: \mathbb{R}^{n} \times \mathcal{U} \to \mathcal{Y}$ is the output function. 

\begin{figure}[H]
\centering
\includegraphics[scale = 0.35]{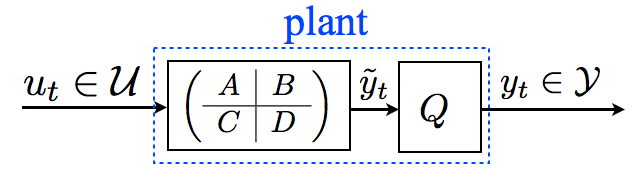} 
\caption{System over finite alphabets with linear internal dynamics \label{Fig:Sofa-LTI}}
\end{figure}

Following a motivating discussion and a set of proposed new definitions for system (\ref{eq:Sofa}), we focus our study in the remainder of this manuscript on special cases where the continuous internal dynamics have the linear structure shown in Figure \ref{Fig:Sofa-LTI} and described by
\begin{subequations}
\label{eq:Sofa-LTI}
\begin{align}
\label{eq:LTI-state}
x_{t+1}  &=  A x_t + B u_t, \\
\label{eq:LTI-out}
\tilde{y}_t  &=  C x_t +  D u_t, \\
\label{eq:Sofa-LTI-qo}
y_t & =  Q(\tilde{y}_t). 
\end{align}
\end{subequations}
As before, $t \in \mathbb{N}$ is the time index, $x_t \in \mathbb{R}^{n}$ is the state, 
$u_t \in \mathcal{U}\subset \mathbb{R}^m$, $|\mathcal{U}|<\infty$, is the input and 
$y_t\in \mathcal{Y} \subset \mathbb{R}^p$, $|\mathcal{Y}|< \infty$, is the quantized output.
$\tilde{y}_t \in \mathbb{R}^p$ is the output of the underlying physical system and matrices $A \in \mathbb{R}^{n \times n}, B \in \mathbb{R}^{n \times m}, C \in \mathbb{R}^{p \times n}$, $ D \in \mathbb{R}^{p \times m}$ are given.
The saturating quantizer $Q: \mathbb{R}^p \to \mathcal{Y} $ is a piecewise-constant function: That is, for any $\tilde{y} \in \mathbb{R}^p$, if $Q$ is continuous at $\tilde{y}$, then there is $\delta > 0$ such that $Q(z) = Q(\tilde{y})$ for all $\|z-\tilde{y}\| < \delta$, $z \in \mathbb{R}^p$ \cite{Caro}.

In particular, when $p=1$ and $Q$ is additionally assumed to be right-continuous, it can be described by:
\begin{equation}
\label{eq:quat1d}
Q((-\infty, \beta_{1})) = y_1, \  Q([\beta_i, \beta_{i+1})) = y_{i+1} \textrm{  for  } i\in \{1,\hdots,|\mathcal{Y}|-1\}, 
\end{equation}
where $\{\beta_i\}_{i=1}^{|\mathcal{Y}|-1}$ are the discontinuous points of $Q$, $\beta_{|\mathcal{Y}|} = \infty$ and $\mathcal{Y}=\{y_1,\hdots,y_{|\mathcal{Y}|}\}$.

\section{Output Observability}
\label{sec:defn-obs}

\subsection{Motivation for a New Notion of Observability}
\label{Sec:Motivation}

A natural starting point in our study is to attempt to apply the definition of LTI system observability 
to systems described by \eqref{eq:Sofa-LTI}. 
Unsurprisingly, we quickly discover that no system in the class  under consideration, or even in the more general class of systems described by \eqref{eq:Sofa},  
is observable in the sense of being amenable to reconstructing its state from an observation of its input and output over some finite time horizon.

\begin{lemma}
\label{prop1}
Consider a system $P$ as in \eqref{eq:Sofa}. 
The initial state of $P$  cannot be uniquely determined by 
knowledge of $(\{u_t\}_{t \in \mathcal{I}}, \{{y}_t\}_{t \in \mathcal{I}})$ over any finite time interval $\mathcal{I} \subset \mathbb{N}$. 
\end{lemma}

\begin{proof}
Let $\mathcal{X}_0$ be the set of all  possible initial states of system (\ref{eq:Sofa}). We have $\mathcal{X}_0 = \mathbb{R}^n $, and hence $\mathcal{X}_0$ is uncountable. 
Now assume that we can uniquely determine any initial condition from the input $u_t$ and output ${y}_t$ over some time interval $\mathcal{I} =  \{0 , \hdots, T\}$ for some $T \in \mathbb{Z}_+$. 
Let $\mathcal{O}$ be the set of all such possible sequences, we have 
$\mathcal{O} \subseteq \mathcal{U}^T \times \mathcal{Y}^T$. 
Since $|\mathcal{U}| < \infty$ and $|\mathcal{Y}| < \infty$,  $\mathcal{U}^T \times \mathcal{Y}^T$ is countable and so is $\mathcal{O}$.
Now by assumption, any initial condition in $ \mathcal{X}_0$ can be uniquely determined by an element in $\mathcal{O}$. 
Equivalently,  there exists a map $\phi: \mathcal{O} \to \mathcal{X}_0$ that is onto. 
This indicates that $\mathcal{X}_0$ is countable (pp. 20, \cite{Caro}),
leading to a contradiction.   
\end{proof}

\begin{rmk} Clearly, Lemma \ref{prop1} holds for system \eqref{eq:Sofa-LTI} which is a special case of system \eqref{eq:Sofa}. Moreover,  
it still holds when the initial state of the system is bounded.  
Specifically, if $ \mathcal{X}_0 = \{x \in \mathbb{R}^n : \|x\| \le b \}$ for some $b \in \mathbb{R}_+$, 
then $\mathcal{X}_0$ is still uncountable and the proof follows unchanged.
\end{rmk}

Lemma \ref{prop1} motivates the  need to think of observability differently for the classes of systems under consideration. 
We propose to shift our focus from the question of ``Can we estimate the state of the system?" (whose answer is clearly no!) to the question of 
``How well can we estimate the output of the system based on our best estimate of the state?". 
Towards that end, we propose in what follows three new notions of \emph{output observability}.

\subsection{Proposed New Notions of Output Observability} 

\begin{figure}[H]
\centering
\includegraphics[scale = 0.3]{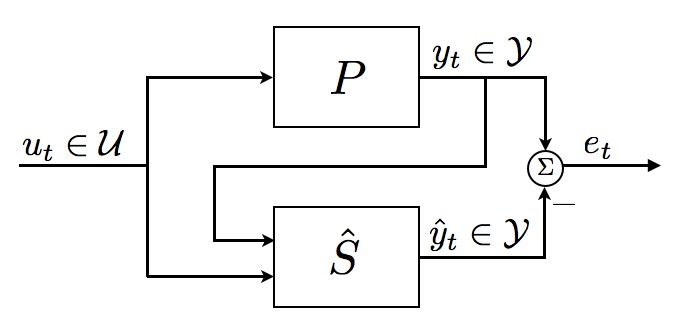} 
\caption{Interconnection of plant and candidate observer. 
\label{Fig:obs-setup}}
\end{figure}

Consider a system over finite alphabet $P$ as defined in \eqref{eq:Sofa} and a discrete-time system $\hat{S}$ as shown in Figure \ref{Fig:obs-setup} and described by: 
\begin{subequations}
\label{eq:observer}
\begin{align}
q_{t+1} &= f(q_t, u_t, {y}_t),\\
\label{eq:obsr_yhat}
\hat{y}_t &= g(q_t, u_t),
\end{align}
\end{subequations}
where $t \in \mathbb{N}$, $q_t \in \mathcal{Q}$ for some state set $\mathcal{Q}$,
$u_t \in \mathcal{U}$, ${y}_t \in \mathcal{Y}$ and $\hat{y}_t \in \mathcal{Y}$. 
 Functions $f: \mathcal{Q} \times \mathcal{U} \times \mathcal{Y} \to \mathcal{Q}$ and $g: \mathcal{Q} \times \mathcal{U} \to \mathcal{Y}$ are given. 
 We say that $\hat{S}$ is a \emph{candidate observer} for $P$.
 Indeed, the setup shown in Figure \ref{Fig:obs-setup} is reminiscent of the classical observer setup. The current state $q_t$ of $\hat{S}$ represents an estimate\footnote{Note that this estimate does not need to live in the state-space of $P$: That is, $\mathcal{Q} \neq \mathbb{R}^n$ in general. For instance, one could be interested in a set-valued estimate of the state as is the case in some of our complementary work \cite{DT11, DT13}.} of the state of $P$ based on observations of its past input and output signals over a (finite) time horizon. 
 The output $\hat{y}_t$ is an estimate of the output of $P$, generated by $\hat{S}$ based on its current state estimate and knowledge of the input: Note in this case that, as is typical in an observer setup, no direct feedthrough from $y_t$ is allowed in (\ref{eq:obsr_yhat}). 
 The error term $e_t$ thus measures the difference between the real output $y_t$ and its estimate $\hat{y}_t$. 
 
 We are now ready to introduce a quantity $\gamma$ that characterizes the quality of a candidate observer $\hat{S}$ as judged by the quality of its estimates of the output, and to propose three associated new notions of output observability of systems over finite alphabets.

\begin{mydef}
\label{def:obsr-gain}
Consider a system $P$ as in \eqref{eq:Sofa} and a candidate observer $\hat{S}$ as in \eqref{eq:observer}, interconnected as shown in Figure \ref{Fig:obs-setup}.  We say $\gamma \in \mathbb{R}_{\ge 0}$ is an \emph{observation gain bound} of the pair $(P, \hat{S})$ if for any $({\bf u}, {\bf y}) \in P$,
\begin{equation}
\label{eq:obsr-gain}
\sup_{T \ge 0} \sum_{t = 0}^T  \| y_t - \hat{y}_t \| - \gamma \|u_t\| < \infty.
\end{equation}
\end{mydef}

Note that $\gamma$ defined in \eqref{eq:obsr-gain} is in accordance with the concept of finite-gain $\mathcal{L}_2$ stability 
\cite{Khalil} of the interconnected system with input $u_t$ and output $e_t$ shown in Figure \ref{Fig:obs-setup}.

\begin{mydef}
Consider a system $P$ as in \eqref{eq:Sofa}. The \emph{$\mathcal{O}$-gain}, $\gamma^*$, of $P$ is defined as:
\begin{equation}
\label{eq:o-gain}
\gamma^* = \inf_{\hat{S} \ \textrm{as in \eqref{eq:observer}} }\Big\{ \gamma \ \textrm{is an observation gain bound of $(P, \hat{S})$}\Big\}.
\end{equation}
\end{mydef}

\begin{mydef} 
\label{def:class-C}
Consider a system $P$ as in \eqref{eq:Sofa}: 
\begin{enumerate}[label=(\alph*)]
\item $P$ is \emph{finite memory output observable} if there exists a candidate observer $\hat{S}$ as in \eqref{eq:observer} and a $T \in \mathbb{Z}_+$ such that for any $({\bf u}, {\bf y}) \in P$, $e_t = 0$ for all $t \ge T$ when $P$ and $\hat{S}$ are interconnected as shown in Figure \ref{Fig:obs-setup}.
\item $P$ is \emph{weakly output observable} if there exists a candidate observer $\hat{S}$ as in \eqref{eq:observer} such that  $\gamma = 0$ is an observation gain bound of $(P, \hat{S})$. 
\item $P$ is \emph{asymptotically output observable} if the $\mathcal{O}$-gain  of $P$ is $\gamma^* = 0$.
\end{enumerate}
\end{mydef}

The three proposed notions satisfy a hierarchy:

\begin{lemma}
\label{lem:ClassCimp}
Consider a system $P$ as in \eqref{eq:Sofa}. We have : $P$ is finite memory output observable $\Rightarrow$ $P$ is weakly output observable $\Rightarrow$ $P$ is asymptotically output observable.
\end{lemma}
\begin{proof}
If $P$ is finite memory output observable, there is a candidate observer $\hat{S}$ such that $y_t - \hat{y}_t = 0$ for all $t \ge T$, and therefore $\sup_{T \ge 0} \sum_{t = 0}^T  \| y_t - \hat{y}_t \| < \infty$. Thus $\gamma = 0$ is an observation gain bound of $(P, \hat{S})$ and $P$ is  weakly output observable. If $P$ is  weakly output observable, then $\gamma^*$ is the infimum of a set of non-negative numbers containing 0, and hence $\gamma^*=0$.
\end{proof}

\begin{rmk}
The converse statements in Lemma \ref{lem:ClassCimp} 
are not necessarily true. Indeed, when $\mathcal{Y}$ is finite,
we have $\gamma = 0$ if and only if for any $({\bf u}, {\bf y}) \in P$, $\hat{y}_t \ne y_t$ at most finitely many times. 
In this case, for every choice of $({\bf u}, {\bf y}) \in P$, $\hat{y}_t=y_t$ for all $t \ge T_{({\bf u}, {\bf y})}$ for some $T_{({\bf u}, {\bf y})} \in \mathbb{Z}_+$, but there may not be a uniform bound on $T_{({\bf u}, {\bf y})}$. Hence $P$ may be weakly output observable but not finite memory output observable. Likewise, we may be able to find candidate observers $\hat{S}$ such that $(P,\hat{S})$ has observation gain bounds arbitrarily close to 0 without being equal to it. Hence $P$ may be asymptotically output observable but not weakly output observable.
\end{rmk}

\color{black}
\section{Finite Memory Output Observability}
\label{sec:FMO}

In this section, we propose (Section \ref{SSec:MainFMO}) and derive (Section \ref{sec:Deri-C1}) a set of conditions characterizing finite memory output observability, and we propose an algorithmic procedure (Section \ref{sec:Algo-C1}) for verifying some of these conditions.

We begin with some relevant definitions and notation. 
Given a system over finite alphabets as in \eqref{eq:Sofa-LTI}, we will use $F({\bf u},t)$ to denote the forced response of the underlying LTI system at time $t$ under input $\bf{u}$, 
$\mathcal{A}$ to denote the set of all the possible values of this forced response,
and $\mathcal{B}$ to denote the set of all the discontinuous points of the quantizer.
That is:
$$F({\bf u},t) = \sum_{\tau = 0}^{t-1} CA^{t-1-\tau}Bu_\tau + Du_t,$$
\begin{equation}
\label{eq:set-A}
\mathcal{A} = \{\alpha \in \mathbb{R}^p : \alpha = F({\bf u},t),  {\bf u} \in \mathcal{U}^{\mathbb{N}}, t \in \mathbb{N} \},
\end{equation}
\begin{equation}
\label{eq:set-B}
\mathcal{B} = \{ \beta \in \mathbb{R}^p: Q(\tilde{y})~ \textrm{is discontinuous at} ~ \tilde{y} = \beta \}.
\end{equation}

\subsection{Main Results}
\label{SSec:MainFMO}
 
\color{black}
We are now ready to propose both necessary conditions and  sufficient conditions for finite memory output observability of system \eqref{eq:Sofa-LTI}. We begin with sufficient conditions.

\begin{theorem}
\label{prop:SC-2-C1}
Consider a system $P$  as in \eqref{eq:Sofa-LTI}.  If $CA^l = {\bf 0}$ for some $l \in \mathbb{Z}_{+}$, then $P$ is finite memory output observable.
\end{theorem}

\begin{theorem}
\label{prop:SC-C1-G}
Consider a system $P$ as in \eqref{eq:Sofa-LTI}, assume that  the initial state $x_0$ is bounded. Let  $\mathcal{E}$ be the collection of generalized eigenvectors of $A$ whose corresponding  eigenvalues have magnitudes   greater than or equal to 1. If $d(\mathcal{A}, \mathcal{B}) \neq 0$, and $\mathcal{E}$ is in the kernel of $C$, then  $P$ is finite memory output observable. 
\end{theorem}

{Intuitively, if the hypothesis in Theorem \ref{prop:SC-2-C1} is satisfied, then the initial state has no impact on the quantized output for large enough time. 
We can therefore determine the output based on past input information, and the system is finite memory output observable.}
Theorem \ref{prop:SC-C1-G} states that if any forced response is at some distance away from the discontinuous points of the quantizer, then the influence of the initial state in the quantized output will eventually disappear, and the knowledge of the past input suffices to predict the output. 
The assumption that $\mathcal{E}$ is in the kernel of $C$ simply means that the (possible) unstable modes of the underlying LTI system do not influence the quantized output.

\begin{rmk}
If the sufficient conditions in Theorem \ref{prop:SC-2-C1} or Theorem \ref{prop:SC-C1-G} are satisfied, then a finite state (or equivalently, finite memory) observer can be constructed to achieve error-free output prediction for large enough time large. We present such a construction (Finite Input Observer Construction) in Section \ref{sec:Deri-C1} . 
\end{rmk}

We next propose  necessary conditions for finite memory output observability. We begin with the case of stable internal dynamics. 

\begin{theorem}
\label{prop:NC-C1}
Consider a system $P$   as in \eqref{eq:Sofa-LTI},   assume that $\rho(A)<1$, $0 \in \mathcal{U}$, and $0 \notin \mathcal{B}$. If $rank(CA^l) = p$ for all $l \in \mathbb{Z}_{+}$, and $\mathcal{A} \cap \mathcal{B} \neq \varnothing$, then $P$ is not finite memory output observable.
\end{theorem}

\begin{theorem}
\label{prop:NC-2-C1}
Consider a system $P$ as in \eqref{eq:Sofa-LTI},  assume that $\rho(A) \ge 1$, $0 \in \mathcal{U}$, and $0 \notin \mathcal{B}$. Define $\mathcal{V} = \{v \in \mathbb{C}^n \setminus 0 : Av = \lambda v, \ \textrm{for some}\  |\lambda| > 1 \}$. If $\mathcal{V}$ is not in the kernel of $C$, and  $Q^{-1}(Q(0))$ is bounded, then system \eqref{eq:Sofa-LTI} is not finite memory output observable. 
\end{theorem}

Intuitively, for Theorem \ref{prop:NC-C1}, if some forced response is exactly at a discontinuous point of the quantizer, then a small perturbation of the initial state can cause a change in the quantized output at  certain time instances, but not at others, and thus the system is not finite memory output observable. Similarly, if the hypotheses in Theorem \ref{prop:NC-2-C1} hold, then under zero input a small perturbation of the initial state around the origin can result in a difference in the  quantized output  at an arbitrarily large time instance, while this perturbation is not reflected in all previous time instances, therefore the system is not finite memory output observable. 

\subsection{Derivation of Main Results}
\label{sec:Deri-C1}

We first derive the sufficient conditions. We begin by proposing a construction for a candidate finite memory observer that will be used in several of the constructive proofs. 

\begin{mydef} \label{def:FIIC}(Finite Input Observer Construction)
Given a system \eqref{eq:Sofa-LTI} and a design parameter $T \in \mathbb{N}$. 
Consider a candidate observer associated with design parameter $T$, $\hat{S}_{T}$, described by
\begin{equation}
\label{eq:ob4LTI}
\begin{aligned}
q_{t+1} &= \phi (q_t, u_t), \\
\hat{y}_t&= \theta (q_t, u_t), \\
q_0 &= q_o,
\end{aligned}
\end{equation}
where  $q_t \in q_o \cup (\bigcup_{i=1}^{T} \mathcal{U}^{i})$ is the state of $\hat{S}$, $u_t \in \mathcal{U}$ is the input of \eqref{eq:Sofa-LTI}.  We enforce that $\hat{S}$ \eqref{eq:ob4LTI} initialize at a \emph{fixed} state: $q_0 = q_o$. Function $\phi: (q_o \cup (\bigcup_{i=1}^{T} \mathcal{U}^{i})) \times \mathcal{U} \to \bigcup_{i=1}^{T} \mathcal{U}^{i}$ is described by: For any $q \in q_o \cup (\bigcup_{i=1}^{T} \mathcal{U}^{i})$, any $u \in \mathcal{U}$,\\
$\rhd$ If $q = q_o$, then $\phi (q,u) = u,$\\
$\rhd$ If $q \in \bigcup_{i=1}^{T-1} \mathcal{U}^{i}$, write $q = (u_1, u_2 \dots u_i)$, then $\phi (q,u) = (u, u_1, u_2 \dots u_i),$\\
$\rhd$ If $q \in \mathcal{U}^{T}$, write $q = (u_1, u_2 \dots u_T)$,  then $\phi (q,u) = (u, u_1, u_2 \dots u_{T-1}).$

The function $\theta: (q_o \cup (\bigcup_{i=1}^{T} \mathcal{U}^{i})) \times \mathcal{U} \to \mathcal{Y}$ is defined as: For any $q \in q_o \cup (\bigcup_{i=1}^{T} \mathcal{U}^{i})$, any $u \in \mathcal{U}$,\\
$\rhd$ If $q \in \mathcal{U}^{T}$, write $q = (\bar{u}_1, \bar{u}_2, \dots, \bar{u}_T)$,  then
\begin{equation}
\label{eq:calalpha}
\theta (q,u) = Q(\sum_{\tau = 1}^{T} CA^{\tau-1}B\bar{u}_\tau + Du), 
\end{equation}
$\rhd$ If $q \notin \mathcal{U}^{T}$, then let $\theta(q,u) = \tilde{y}_\varnothing$ for some $\tilde{y}_\varnothing \in \mathcal{Y}$.
\hfill \qedsymbol
\end{mydef}

We are now ready to prove our first result:

\begin{proof} (Theorem \ref{prop:SC-2-C1})
Given system \eqref{eq:Sofa-LTI}, recall the form of $\tilde{y}_t$ in terms of linear dynamics and note $CA^l = {\bf 0}$, we see that for all $t \ge l$: 
$\tilde{y}_t = \sum_{\tau = t - l}^{t-1} CA^{t-1-\tau}Bu_\tau + Du_t. $
Consider an observer $\hat{S}$ constructed according to Definition \ref{def:FIIC} with parameter $T = l$. Then 
$\hat{y}_t = Q(\sum_{\tau = t - l}^{t-1} CA^{t-1-\tau}Bu_\tau + Du_t) = Q(\tilde{y}_t) = y_t, \ \forall  \ t \ge l.$  
We conclude that system \eqref{eq:Sofa-LTI} is finite memory output observable.
\end{proof}

We next establish several observations that will be instrumental in deriving the remaining results.

\begin{lemma}
\label{lem:ISSB}
Consider system $P$ as in \eqref{eq:Sofa-LTI}, assume that $\rho(A)<1$ and that the initial state $x_0$ is bounded. There exists a $b_0 \in \mathbb{R}_{+}$ such that $\|x_t\| \le b_0$ for all $t \in \mathbb{N}$. 
\end{lemma}
\begin{proof}
The solution associated with initial condition $x_0$ is given by 
$x_t = A^t x_0 + \sum_{\tau = 0}^{t-1} A^{t-1-\tau} B u_\tau.$ 
For any $t \in \mathbb{N}$, we have
$\|x_t\| = \|A^t x_0 + \sum_{\tau = 0}^{t-1} A^{t-1-\tau} B u_\tau\|
 \le \max \{\|x_0\|, \|B u_t\|\} \sum_{\tau = 0}^{t} \|A^{\tau}\|.$ 

Since $\sum_{\tau = 0}^{\infty} \|A^{\tau}\|$ converges (pp. 299, \cite{Horn}), we can find an upper bound $b_1 \in \mathbb{R}_+$ such that $\sum_{\tau = 0}^{\infty} \|A^{\tau}\| \le b_1$. Since $u_t \in \mathcal{U}$ and $\mathcal{U}$ is finite, $\|B u_t\|$ is also bounded. Let $b_2 = \max\{\|x_0\|, \max\{\|B u\|: u \in \mathcal{U}\}\}$, and then let $b_0 = b_1b_2$, we have $\|x_t\| \le b_0$ for all $t \in \mathbb{N}$. 
\end{proof}

\begin{lemma}
\label{lem:tytindab2b}
Consider system $P$  as in \eqref{eq:Sofa-LTI}, assume that $\rho(A)<1$, $C \neq {\bf 0}$ and that the initial state $x_0$ is bounded.  If $d(\mathcal{A}, \mathcal{B}) \neq 0$, then there exists a $T \in \mathbb{Z}_+$ such that for all $t \ge T$ :
\begin{equation}
\label{eq:ytinball}
\tilde{y}_t \in B_{\frac{d(\mathcal{A}, \mathcal{B})}{2}}(\alpha),
\end{equation} 
where $\alpha = \sum_{\tau = t - T}^{t-1} CA^{t-1-\tau}Bu_\tau + Du_t$ and $\alpha \in \mathcal{A}$. 
\end{lemma}

\begin{proof}
By Lemma \ref{lem:ISSB}, $\|x_t\| \le b_0$ for all $t \in \mathbb{N}$ for some $b_0 \in \mathbb{R}_{+}$. Since $\lim_{\tau \to \infty} A^{\tau} = 0$ (pp. 298, \cite{Horn}), recall the assumption that $C \neq {\bf 0}$, we can choose $T \in \mathbb{Z}_+$ such that 
\begin{equation}
\label{eq:chooseT-C1}
\|A^T\| < \frac{d(\mathcal{A}, \mathcal{B})} {2  b_0 \|C\|}.
\end{equation}
Then $\|CA^T x_t\|  \le \|C\| \|A^T\| \|x_t\| < d(\mathcal{A}, \mathcal{B})/2$ for all $t \in \mathbb{N}$. Recall that $\tilde{y}_t = CA^T x_{t-T} + \sum_{\tau = t - T}^{t-1} CA^{t-1-\tau}Bu_\tau + Du_t$, we see \eqref{eq:ytinball} holds.
 \end{proof}

Next, we observe that the quantized output $y_t$ can be determined by the knowledge of the forced response.
\begin{lemma}
\label{lem:prop-of-Q}
Consider system $P$  as in \eqref{eq:Sofa-LTI}, and sets $\mathcal{A}$ and $\mathcal{B}$ as in \eqref{eq:set-A} and \eqref{eq:set-B} respectively. If $d(\mathcal{A}, \mathcal{B}) \neq 0$, then  for any $\alpha \in \mathcal{A}$, 
\begin{equation}
\label{eq:prop-of-Q}
y \in cl(B_{\frac{d(\mathcal{A}, \mathcal{B})}{2}}(\alpha)) \ \Rightarrow \ Q(y) = Q(\alpha). 
\end{equation}
\end{lemma}

\begin{proof}
First observe that $Q$ is continuous at any point $y \in cl(B_{\frac{d(\mathcal{A}, \mathcal{B})}{2}}(\alpha))$, otherwise $d(\mathcal{A}, \mathcal{B}) \le d(\mathcal{A}, \mathcal{B})/2$, which contradicts with $d(\mathcal{A}, \mathcal{B})>0$. Next, assume there is a $y \in cl(B_{\frac{d(\mathcal{A}, \mathcal{B})}{2}}(\alpha))$ such that $Q(y) \neq Q(\alpha)$. Define two sequences $\{w_n\}_{n=1}^\infty$, $\{v_n\}_{n=1}^\infty$ as follows: Let $w_1 = \alpha$, $v_1 = y$. For any $n \ge 2$, let $z = (w_{n-1} + v_{n-1})/2$, if $Q(z) \neq Q(w_{n-1})$,  let $w_{n} = w_{n-1}, v_{n} = z$; otherwise, let $w_{n} = z, v_{n} = v_{n-1}$. By this definition, we see that $Q(w_n) \neq Q(v_n)$ implies $Q(w_{n+1}) \neq Q(v_{n+1})$. Since $Q(w_1) \neq Q(v_1)$, by induction, we have: $Q(w_n) \neq Q(v_n)$ for all $n \in \mathbb{Z}_+$. At the same time, it is clear that $\|w_n - v_n\| = (1/2)^n\|w_1-v_1\|$. Note that $\{w_n\}_{n=1}^\infty, \{v_n\}_{n=1}^\infty \subset cl(B_{\frac{d(\mathcal{A}, \mathcal{B})}{2}}(\alpha))$ and $cl(B_{\frac{d(\mathcal{A}, \mathcal{B})}{2}}(\alpha))$ is a compact set in $\mathbb{R}^p$, by renaming, there are subsequences $\{w'_n\} \subset \{w_n\}$ and $\{v'_n\} \subset \{v_n\}$  such that  $Q(w'_n) \neq Q(v'_n)$, $\|w'_n - v'_n\| \le (1/2)^n\|w_1-v_1\|$, for all $n \in \mathbb{Z}_+$, and $\lim_{n \to \infty} w'_n = w$, $\lim_{n \to \infty} v'_n = v$ for some $w, v$ in $cl(B_{\frac{d(\mathcal{A}, \mathcal{B})}{2}}(\alpha))$.
Since $\|w-v\| \le \|w-w'_n\| + \|w'_n-v'_n\| + \|v'_n-v\|$, we see that for any $\epsilon >0$, $\|w-v\|<\epsilon$, and consequently $w = v$.
Since $w \in cl(B_{\frac{d(\mathcal{A}, \mathcal{B})}{2}}(\alpha))$, $Q$ is continuous at $w$. Recall $Q$ is piecewise-constant, there is $\delta >0$ such that $Q(z) = Q(w)$ for all $\|z-w\| < \delta$, $z \in \mathbb{R}^p$. Since $\lim_{n \to \infty} w'_n = w$, there is $N_1$ such that $Q(w'_n) = Q(w)$, for all $n \ge N_1$. Similarly, there is $N_2$ such that $Q(v'_n) = Q(v) = Q(w)$ for all $n \ge N_2$. Let $n = \max\{N_1, N_2\}$, then $Q(w'_n) = Q(w) = Q(v'_n)$, which contradicts with $Q(w'_n) \neq Q(v'_n)$ for all $n \in \mathbb{Z}_+$. Therefore, assumption is false, and we conclude that for any $\alpha \in \mathcal{A}$ and any $y \in cl(B_{\frac{d(\mathcal{A}, \mathcal{B})}{2}}(\alpha))$, $Q(y) = Q(\alpha)$.
\end{proof}

Given a system \eqref{eq:Sofa-LTI}, we first decompose the state $x_t$ into stable modes and unstable modes, and we make an observation on the stable modes. In particular, consider the Jordan canonical form of the matrix $A$,
\begin{equation}
\label{eq:JordecompA}
A = M J M^{-1},
\end{equation}
where matrix $J$ is in partitioned diagonal form, and matrix $M$ is a generalized modal matrix for $A$ (pp. 205, \cite{Bronson}).
Write $M = [v_1 \ v_2 \ \cdots \ v_n]$, where $v_i \in \mathbb{C}^n$ for $1 \le i \le n$, then each $v_i$ is a generalized eigenvector of $A$, and $\{v_i\}_{i=1}^n$ form a basis of $\mathbb{R}^n$. For each $v_i$, use $\lambda_i$ to denote the eigenvalue of $A$ corresponding to $v_i$. 
Next,  we decompose the state vector $x_t$  using $\{v_i\}_{i=1}^n$.  For all $t \in \mathbb{N}$, write $x_t$
 as a linear combination of $\{v_i\}_{i=1}^n$,
 \begin{equation}
 \label{eq:xttoliges}
 x_t = \sum_{i=1}^n [\alpha_t]_i v_i,
 \end{equation}
 where $\alpha_t \in \mathbb{C}^n$ is the coordinates of $x_t$ corresponding to the basis $\{v_i\}_{i=1}^n$. Here $[\alpha_t]_i, 1 \le i \le n$ are the coordinates of $x_t$ with respect to the basis $\{v_i\}_{i=1}^n$, and $[\alpha_t]_i$ with $v_i \notin \mathcal{E}$ are the coordinates corresponding with the stable generalized eigenvectors. 
We make an observation on these stable modes in the following.
 
\begin{lemma} 
\label{lem:Jordisss}
Consider system \eqref{eq:Sofa-LTI}, write $x_t$ as in \eqref{eq:xttoliges}, then $\sum_{i:v_i \notin \mathcal{E}}   |[\alpha_t]_i| < b_1, \forall \ t \in \mathbb{N}$ for some $b_1 \in \mathbb{R}_+$. 
\end{lemma}

\begin{proof}
Let $M$ be given as in \eqref{eq:JordecompA}, and note that $M = [v_1 \ v_2 \ \cdots \ v_n]$, recall \eqref{eq:xttoliges}, we have $M^{-1} x_t = \alpha_t$.  Define $\alpha_t^U \in \mathbb{C}^n$ as: If $v_i \in \mathcal{E}$, then $ [\alpha_t^U]_i = [\alpha_t]_i$; otherwise  $[\alpha_t^U]_i = 0$.  
 Similarly, define $\alpha_t^S \in \mathbb{C}^n$ as: If $v_i \notin \mathcal{E}$, then $[\alpha_t^S]_i = [\alpha_t]_i$; otherwise  $[\alpha_t^S]_i = 0$. 
We see that $\alpha_t = \alpha_t^U + \alpha_t^S$. Essentially, $\alpha_t^U$ ($\alpha_t^S$) are the coordinates corresponding with the unstable (stable) generalized eigenvectors.  

Similarly, we decompose the $Bu_t$ term in system \eqref{eq:Sofa-LTI}. For all $t \in \mathbb{N}$, write $Bu_t$
 as a linear combination of $\{v_i\}_{i=1}^n$,
$ Bu_t = \sum_{i=1}^n [\beta_t]_i v_i,$
where $\beta_t \in \mathbb{C}^n$ is the coordinates of $Bu_t$ corresponding to the basis $\{v_i\}_{i=1}^n$. Then $M^{-1}Bu_t = \beta_t$. 
Define $\beta_t^U \in \mathbb{C}^n$ as: If $v_i \in \mathcal{E}$, then $ [\beta_t^U]_i = [\beta_t]_i$; otherwise  $[\beta_t^U]_i = 0$,  
and define $\beta_t^S \in \mathbb{C}^n$ as: If $v_i \notin \mathcal{E}$, then $[\beta_t^S]_i = [\beta_t]_i$; otherwise  $[\beta_t^S]_i = 0$.  We also have $\beta_t = \beta_t^U + \beta_t^S$.
 
 Recall \eqref{eq:Sofa-LTI} and $A = MJM^{-1}$, we have
 $
 M M^{-1} x_{t+1} = MJM^{-1} x_t + M M^{-1} Bu_t.
 $
 Recall $M^{-1} x_t = \alpha_t$, and $M^{-1}Bu_t = \beta_t$, we have
$
   \alpha_{t+1} = J \alpha_t +  \beta_t.
 $
 Consequently 
 \begin{equation}
 \label{eq:jordecomus}
 \begin{aligned}
   \alpha_{t+1}^U + \alpha_{t+1}^S 
   & =  (J \alpha_t^U + \beta_t^U) +  (J \alpha_t^S + \beta_t^S). 
   \end{aligned}
 \end{equation} 
 Consider the term $J \alpha_t^U$, and write $J = [w_1 \cdots w_n]$ where $w_1, \dots, w_n \in \mathbb{C}^n$, then
 $
 J \alpha_t^U = \sum_{i=1}^n [\alpha_t^U]_i w_i.
 $
 Since $[\alpha_t^U]_i = 0$ for all $i$ such that $v_i \notin \mathcal{E}$, we have
 \begin{equation}
 J \alpha_t^U = \sum_{i : v_i \in \mathcal{E}} [\alpha_t^U]_i w_i.
 \end{equation}
 
 Recall the definition of $\lambda_i$ for $i = 1,\dots,n$, and the form of $J$, we see that if $\lambda_j \neq \lambda_i$, then $[w_i]_j = 0$, for all $1\le i,j \le n$. For any $i $ such that $v_i \in \mathcal{E}$, and any $j$ such that $v_j \notin \mathcal{E}$, we have $|\lambda_j| <1$ and $|\lambda_i| \ge 1$, therefore $\lambda_j \neq \lambda_i$, and consequently $[w_i]_j = 0$. We see that for all $j$ such that $v_j \notin \mathcal{E}$,
$
 [J \alpha_t^U]_j = 0.
$
 Similarly, for all $j$ such that $v_j \in \mathcal{E}$,
$
 [J \alpha_t^S]_j = 0.
$
 Recall \eqref{eq:jordecomus}, we see that 
$ \alpha_{t+1}^S  =   (J \alpha_t^S + \beta_t^S). 
$ 
Consider the term $J \alpha_t^S$, we have
$ J \alpha_t^S = \sum_{i : v_i \notin \mathcal{E}} [\alpha_t^S]_i w_i.
$
 Define a square matrix $J^S$ to be $J^S = [w_1^S \ w_2^S \cdots \ w_n^S]$, where
 \begin{equation*}
  w_i^S = \left\{ \begin{array}{ll}
 w_i, & \textrm{if} \ v_i \notin \mathcal{E}, \\
 0, & \textrm{otherwise.}
 \end{array}
 \right.
 \end{equation*}
 Note that $J^S$ is Schur-stable.
 Then we have 
$ J \alpha_t^S = \sum_{i : v_i \notin \mathcal{E}} [\alpha_t^S]_i w_i =  \sum_{i =1}^n [\alpha_t^S]_i w_i^S = J^S \alpha_t^S.
$ And consequently, we have
 \begin{equation}
 \label{eq:stablemodes}
 \alpha_{t+1}^S  =   J^S \alpha_t^S + \beta_t^S,
 \end{equation}
 where $J^S$ is Schur-stable. 
 Since $\|\alpha_0^S\|_1  \le \|M^{-1}\|_1 \| x_0\|_1$, and $x_0$ is bounded, we see that $\|\alpha_0^S\|_1$ is bounded. Similarly, note that  $\mathcal{U}$ is a finite set in $\mathbb{R}^m$, we see that $\|\beta_t^S\|_1$ is uniformly bounded. Given system  \eqref{eq:stablemodes}, since $J^S$ is Schur-stable, $\|\alpha_0^S\|_1$ is bounded, and $\|\beta_t^S\|_1$ is uniformly bounded, by the derivation of Lemma \ref{lem:ISSB}, $\|\alpha_t^S\|_1 < b_1$ for some $b_1 \in \mathbb{R}_+$ for all $t \in \mathbb{N}$.  Note that $\sum_{i:v_i \notin \mathcal{E}}    |[\alpha_t]_i| = \|\alpha_t^S\|_1$, we have
 $\sum_{i:v_i \notin \mathcal{E}}   |[\alpha_t]_i|  < b_1,  \ b_1 \in \mathbb{R}_+, $ for all $t \in \mathbb{N}$. 
\end{proof}

Next, we make an observation about the forced response of the underlying linear dynamics.

\begin{lemma}
\label{lem:tytindab2bunst}
Consider a system $P$  as in \eqref{eq:Sofa-LTI}, assume the hypotheses in Theorem \ref{prop:SC-C1-G} are satisfied.  Then there exists a $T \in \mathbb{Z}_+$ such that for all $t \ge T$ :
\begin{equation}
\label{eq:ytinballunst}
\tilde{y}_t \in B_{\frac{d(\mathcal{A}, \mathcal{B})}{2}}(\alpha),
\end{equation} 
where $\alpha = \sum_{\tau = t - T}^{t-1} CA^{t-1-\tau}Bu_\tau + Du_t$ and $\alpha \in \mathcal{A}$. 
\end{lemma}

\begin{proof}
Use $\{e_i\}_{i = 1}^n$ to denote the standard basis of $\mathbb{R}^n$, and recall the computation of powers of a Jordan block (pp. 57, \cite{Hespanha}), then for all $t\ge n$,
$
A^t v_i 
= M J^t e_i 
= \sum_{j: \lambda_j = \lambda_i} \lambda_i^t \ p_{(j, i)}(t) \ v_j,
$
where $p_{(j, i)}(t)$ is some polynomial in $t$ that depends on the pair $(j,i)$. Recall the particular form of $J_i^t$, the upper triangular elements of $J_i^t$ has the form $\frac{t!}{k! (t-k)!} \cdot \lambda_i^{t-k}$, where $0 \le k \le n_i - 1$, and $n_i$ corresponds to the size of $J_i$. Note that $\frac{t!}{k! (t-k)!} \le t^n$, and that $|\lambda_i|^{t-k} \le \kappa |\lambda_i|^t$ for some $\kappa \in \mathbb{R}_+$: If $|\lambda_i| \ge 1$ or $|\lambda_i| = 0$, let $\kappa_1 = 1$; if $1 > |\lambda_i| >0$, let $\kappa_2 = (\max\{1/|\lambda_j|: 0 < |\lambda_j| < 1, 1 \le j \le n\})^n$; take $\kappa = \max\{\kappa_1, \kappa_2\}$. Combine these observations, we conclude that for any $i \in \{1,\dots,n\}$, and any $t \ge n$,
\begin{equation}
\label{eq:Apowerliges}
A^t v_i = \sum_{j: \lambda_j = \lambda_i} \lambda_i^t \ p_{(j, i)}(t) \  v_j,
\end{equation}
where 
\begin{equation}
\label{eq:polybdd}
|\lambda_i^t \ p_{(j, i)}(t)| \le  \kappa \cdot t^n  |\lambda_i|^t,
\end{equation}
for some $\kappa \in \mathbb{R}_+$.

For any $T \ge n$, and any $t \in \mathbb{N}$, recall \eqref{eq:xttoliges}, \eqref{eq:Apowerliges}, we have
\begin{equation}
\label{eq:CAtx02p}
\begin{aligned}
CA^T x_t = &\sum_{i:v_i \in \mathcal{E}} [\alpha_t]_i  \sum_{j: \lambda_j = \lambda_i} \lambda_i^T p_{(j, i)}(T) C v_j \\
&+  \sum_{i:v_i \notin \mathcal{E}} [\alpha_t]_i  \sum_{j: \lambda_j = \lambda_i} \lambda_i^T p_{(j, i)}(T) C v_j. 
\end{aligned}
\end{equation} 
If $v_i \in \mathcal{E}$ and $\lambda_j = \lambda_i$, then $v_j \in \mathcal{E}$. Since $\mathcal{E}$ is in the kernel of $C$, for any $i$ such that  $v_i \in \mathcal{E}$, $Cv_j = 0$ for all $j$ such that $\lambda_j = \lambda_i$. Therefore 
$
\sum_{i:v_i \in \mathcal{E}} [\alpha_t]_i  \sum_{j: \lambda_j = \lambda_i} \lambda_i^T p_{(j, i)}(T) C v_j  = 0.
$
Continued from \eqref{eq:CAtx02p}, we have 
$
CA^T x_t =  \sum_{i:v_i \notin \mathcal{E}} [\alpha_t]_i  \sum_{j: \lambda_j = \lambda_i} \lambda_i^T p_{(j, i)}(T) C v_j. 
$
Recall \eqref{eq:polybdd}, we have
\begin{equation}
\begin{aligned}
\|CA^T x_t\|  \le  \sum_{i:v_i \notin \mathcal{E}} \kappa n \eta \cdot  |[\alpha_t]_i| \ T^n  \ |\lambda_i|^T, 
\end{aligned}
\end{equation}
where $\eta = \max\{ \|C v_j \|: 1 \le j \le n\}$. For any $i$ such that $v_i \notin \mathcal{E}$, $|\lambda_i|<1$. Let $\rho = \max\{|\lambda_i| : v_i \notin \mathcal{E}\}$, then $0 \le \rho < 1$. Consequently,
$
\|CA^T x_t\|  \le \sum_{i:v_i \notin \mathcal{E}} \kappa n \eta \cdot  |[\alpha_t]_i| T^n  \rho^T  \le \kappa n \eta \cdot T^n  \rho^T \sum_{i:v_i \notin \mathcal{E}}   |[\alpha_t]_i|.
$
By Lemma \ref{lem:Jordisss}, there is $b_1 > 0$ such that $ \sum_{i:v_i \notin \mathcal{E}}   |[\alpha_t]_i| < b_1$. 
We have
\begin{equation}
\|CA^T x_t\| \le b_1 \kappa n \eta   \cdot T^n  \rho^T,
\end{equation}
for any $T \ge n$, and any $t \in \mathbb{N}$. 
Note that $0 \le \rho <1$, therefore $\lim_{T \to \infty} T^n  \rho^T = 0$. 
Choose $T \in \mathbb{Z}_+$ such that  
$
T^n  \rho^T < \frac{d(\mathcal{A}, \mathcal{B})} {2  b_1 \kappa n \eta}, 
$
 we have 
$
\|CA^T x_t\| < d(\mathcal{A}, \mathcal{B})/2,
$
for all $t \in \mathbb{N}$. Then for all $t \ge T$,  
$\tilde{y}_t = CA^Tx_{t-T} + \sum_{\tau = t - T}^{t-1} CA^{t-1-\tau}Bu_\tau + Du_t
 \in B_{\frac{d(\mathcal{A}, \mathcal{B})}{2}}(\alpha),
$ and consequently \eqref{eq:ytinballunst} holds.

\end{proof}

We are now ready to show Theorem \ref{prop:SC-C1-G}.

\begin{proof} (Theorem \ref{prop:SC-C1-G}) \\
{\it Case 1: $A$ Hurwitz}.
First note that if $C$ is the zero matrix, then $y_t = Q(Du_t)$. Since $y_t$ can be determined by the knowledge of $u_t$, system \eqref{eq:Sofa-LTI} is (C1). Therefore in the following derivation, we only consider the case $C \neq {\bf 0}$.
Recall Lemma \ref{lem:tytindab2b}, let $T$ be  determined by \eqref{eq:chooseT-C1} such that \eqref{eq:ytinball} holds. Let $\hat{S}$ be constructed according to Definition \ref{def:FIIC} with this parameter $T$. 
Since $q_0 = q_o$, we have
\begin{equation}
\label{eq:storepastinput}
q_t = (u_{t-1}, u_{t-2}, \dots, u_{t-T}), \ \forall \ t \ge T.
\end{equation} where $\{u_t\}$ is the input of system \eqref{eq:Sofa-LTI}.

For any $t \ge T$, recall \eqref{eq:ob4LTI},  \eqref{eq:calalpha}, \eqref{eq:storepastinput}, we have 
$\hat{y}_t =  Q(\sum_{\tau = 1}^{T} CA^{\tau-1}B{u}_{t - \tau} + Du_t) 
= Q(\sum_{\tau = t-T}^{t-1} CA^{t- \tau -1}B{u}_{\tau} + Du_t) = Q(\alpha),$ 
where $\alpha = \sum_{\tau = t - T}^{t-1} CA^{t-1-\tau}Bu_\tau + Du_t$.
Recall Lemma \ref{lem:tytindab2b}, we have $\tilde{y}_t \in B_{\frac{d(\mathcal{A}, \mathcal{B})}{2}}(\alpha) \subset cl(B_{\frac{d(\mathcal{A}, \mathcal{B})}{2}}(\alpha))$. By Lemma \ref{lem:prop-of-Q}, we have $Q(\tilde{y}_t) = Q(\alpha)$. We conclude that $\hat{y}_t = Q(\alpha) = Q(\tilde{y}_t) = y_t$ for all $t \ge T$, and consequently system \eqref{eq:Sofa-LTI} is finite memory output observable.

{\it Case 2: $A$ is not Hurwitz}.
By Lemma \ref{lem:tytindab2bunst}, there exists a $T \in \mathbb{Z}_+$ such that for all $t \ge T$ :
$
\tilde{y}_t \in B_{\frac{d(\mathcal{A}, \mathcal{B})}{2}}(\alpha),
$ 
where $\alpha = \sum_{\tau = t - T}^{t-1} CA^{t-1-\tau}Bu_\tau + Du_t$ and $\alpha \in \mathcal{A}$. 
The rest of this derivation follows the exact same lines of the derivation for the case where $A$ is Hurwitz.
\end{proof}

We now shift our focus to deriving the necessary conditions.

\begin{proof} (Theorem \ref{prop:NC-C1})
The proof is by contradiction. 
Since $\mathcal{A} \cap \mathcal{B} \neq \varnothing$, there exist $t_1 \in \mathbb{N}$ and $\bf{u}^1$ such that $t_1 = \min\{t: F({\bf u},t) \in \mathcal{A} \cap \mathcal{B}\}$ and $F(u^1, t_1) \in \mathcal{B}$. The existence of the minimum is guaranteed by the well-ordering principle of nonnegative integers (pp. 28, \cite{Nicholson}).  $t_1$ being a minimum indicates that $F({\bf u^1},t)$ is not in $\mathcal{B}$ for any $t < t_1$. So we can define the following distance:
\begin{equation}
\label{eq:defnd1}
d_1 = \left\{ \begin{array}{ll}
d(\{0\}, \mathcal{B}),  & \textrm{if} \ t_1 = 0\\
d(\{0\} \cup \{F({\bf u^1},t): 0 \le t \le t_1 - 1\}, \mathcal{B}),  & \textrm{if} \ t_1 \ge 1
\end{array} \right.
\end{equation}
The definition of $t_1$ and $0 \notin \mathcal{B}$ imply $d_1 > 0$.

Assume that system \eqref{eq:Sofa-LTI} is finite memory output observable, than there exists an observer $\hat{S}$ \eqref{eq:observer} and $T$ such that for any $x_0 \in \mathbb{R}^n$, any ${\bf u} \in \mathcal{U}^{\mathbb{N}}$, $ \hat{y}_t = y_t $ for all $t \ge T$. Without loss of generality, we assume that $T \ge t_1$ (if $T< t_1$, just let $T = t_1$, then $ \hat{y}_t = y_t $ for all $t \ge T$ still holds).

Construct an input sequence $\bf{u}$ of system \eqref{eq:Sofa-LTI}. Given $\bf{u}^1$, use the truncated sequence of $\bf{u}^1$: $  \{u^1_t: 0 \le t \le t_1\}$,  the input sequence $\bf{u}$ is described as follows:
\begin{equation}
\label{eq:input-NC-C1}
u_t = \left\{ \begin{array}{ll}
 0, & 0 \le t \le T - t_1 - 1\\
 u^1_{t - (T - t_1 )}, & T - t_1  \le t \le T\\
 0, & t > T
  \end{array} \right.
\end{equation}
Basically we insert the truncated sequence of $  \{u^1_t: 0 \le t \le t_1\}$ into a  zero input. 
If distinct initial states $x^1_0$ and $x^2_0$ satisfy: 
\begin{equation}
\label{eq:necon1}
\|CA^t x^i_0\| < d_1/2,\qquad i = 1,2 
\end{equation}
for $t = 0,1 \cdots T-1$, then under input $\bf{u}$ \eqref{eq:input-NC-C1}, the corresponding outputs of the underlying LTI system, $\tilde{y}_t^1$
and $\tilde{y}_t^2$, satisfy:
$\tilde{y}_t^i \in B_{{d_1}/{2}}(\alpha), i = 1,2$,  
for some $\alpha \in \{0\} \cup \{F({\bf u^1},t): 0 \le t \le t_1 - 1\}$, for $t = 0,1 \cdots T-1$. Recall the definition of $d_1$ and Lemma \ref{lem:prop-of-Q}, we have $Q(\tilde{y}_t^1) = Q(\alpha) = Q(\tilde{y}_t^2)$. Consequently, we have
$y_t^1 = y_t^2, \ t = 0,1 \cdots T-1,$  
where $y_t^i$ is the output of system \eqref{eq:Sofa-LTI} when the initial state is $x_0^i$ and the input is $\bf{u}$ \eqref{eq:input-NC-C1}.

In addition, since $Q$ is not continuous at $F({\bf u},t) = F(u^1, t_1)$, for any $\delta > 0$, there is $z \in \mathbb{R}^p$ such that $Q(z + F({\bf u},t)) \neq Q(F({\bf u},t))$, and $\|z\|<\delta$. Since $rank(CA^T) = p$ by assumption, write $CA^T = [v_1 \ v_2 \ \cdots \ v_n]$, where $v_1, \dots, v_n \in \mathbb{R}^p$, then there is $\{i_1, i_2, \dots, i_p\} \subset \{1, 2, \dots, n\}$ such that $[v_{i_1} \ v_{i_2} \ \cdots \ v_{i_p}]$ is invertible. Let $V = [v_{i_1} \ v_{i_2} \ \cdots \ v_{i_p}]$, and let $K_A = \sup\{\|A^t\| : t = 0,1,2,\cdots\}$. choose $\delta = \frac{d_1}{2 \|V^{-1}\| \|C\| K_A} > 0$, 
then there is $z \in \mathbb{R}^p$ such that $Q(z + F({\bf u},t)) \neq Q(F({\bf u},t))$, and $\|z\|<\delta$. Let $w = V^{-1}z$, and write $w = [w_1 \ w_2 \ \cdots \ w_p]^T \in \mathbb{R}^p$, define  $x^* = [x_1^* \ x_2^* \ \cdots \ x_n^*]^T \in \mathbb{R}^n$ as: For all $1 \le j \le p$, $x_{i_j}^* = w_j$; for all $1 \le l \le n$ and $l \notin \{i_1, i_2, \dots, i_p\}$, $x_l^* = 0$. Then we have $CA^T x^* = Vw$, and $\|x^*\| = \|w\|$.

Consider two distinct initial conditions: $x_0^1 = 0, x_0^2 = x^*$. Then for all $t = 0,1 \cdots T-1$, $\|CA^t x^2_0\| \le   \|C\| K_A \|w\|  < \|C\| K_A \|V^{-1}\| \cdot \frac{d_1}{2 \|V^{-1}\| \|C\| K_A} = d_1/2$. Therefore \eqref{eq:necon1} holds, and consequently $y_t^1 = y_t^2, \ t = 0,1 \cdots T-1$. At $t = T$, $y_T^1 = Q(F({\bf u},t))$, and $y_T^2 = Q(CA^Tx^*+F({\bf u},t)) = Q(z + F({\bf u},t)) \neq Q(F({\bf u},t))$, therefore $y_T^1 \neq y_T^2$.

Since system \eqref{eq:Sofa-LTI} is assumed to be finite memory output observable, let $\hat{y}_t^1$ and $\hat{y}_t^2$ be the output of the corresponding $\hat{S}$ when the input is $\bf{u}$ \eqref{eq:input-NC-C1} and  initial conditions are  $x_0^1, x_0^2 $ respectively.
Then at $t = T$, recall \eqref{eq:observer}, we have $\hat{y}_T^1 = g(f(\dots f(f(q_0, {u}_0, y_0^1), {u}_1, y_1^1)\dots, {u}_{T-1}, y_{T-1}^{1}), {u}_T)$, and  $\hat{y}_T^2 = g(f(\dots f(f(q_0, {u}_0, y_0^2), {u}_1, y_1^2)\dots, {u}_{T-1}, y_{T-1}^{2}), {u}_T)$. Recall that $y_t^1 = y_t^2, \ t = 0,1 \cdots T-1$, we have $\hat{y}_T^1 = \hat{y}_T^2$. Since $y_T^1 \neq y_T^2$, 
there is $i \in \{1,2\}$ such that $\hat{y}_T^i \neq y_T^i$. This is a contradiction with system \eqref{eq:Sofa-LTI} being finite memory output observable.
\end{proof}

\begin{proof} (Theorem \ref{prop:NC-2-C1})
Since $\mathcal{V}$ is not in the kernel of $C$, there is $v \in \mathcal{V}$ such that $Cv \neq 0$. Without loss of generality, let $\|Cv\|_1 = 1$.
Since $v \in \mathcal{V}$, we have $Av = \lambda v$ for some $|\lambda| > 1$, $\lambda \in \mathbb{C}$.  Next, we define a set $\mathcal{O}$ as
\begin{equation}
\label{eq:seto}
\mathcal{O} = \{ \alpha \in \mathbb{R}_+ | Re(\gamma Cv) \in Q^{-1}(Q(0)), \ \forall \ |\gamma| \le \alpha, \gamma \in \mathbb{C}\}.   
\end{equation}

Next, we show that $\mathcal{O}$ is non-empty and bounded. Write $Cv = [v_1 \ v_2 \ \dots \ v_n]^T$, where $v_1, \dots, v_p \in \mathbb{C}$ and $|v_1| + \dots + |v_p| = 1$. For any $\gamma \in \mathbb{C}$, we have $|Re(\gamma v_i)| \le |\gamma\|v_i|$, therefore
$
\|Re(\gamma Cv) \|_1 = \sum_{i=1}^p |Re(\gamma v_i)| \le |\gamma| \sum_{i=1}^p |v_i| = |\gamma|.
$
Since $Q$ is a piecewise constant function, and $0 \notin \mathcal{B}$, there is $r>0$ such that $B_r(0) \subset Q^{-1}(Q(0))$, where $B_r(0) = \{x \in \mathbb{R}^p : \|x\|_1 < r\}$. Therefore, for all $\gamma$ with $|\gamma| \le r/2$, $Re(\gamma Cv) \in B_r(0)$. Consequently $r/2 \in \mathcal{O}$, and $\mathcal{O}$ is nonempty.

Next, we show that $\mathcal{O}$ is bounded. Since $Q^{-1}(Q(0))$ is bounded by assumption, let $Q^{-1}(Q(0)) \subset B_\sigma (0)$ for some $\sigma >0$. Since $Cv = [v_1 \ v_2 \ \dots \ v_n]^T \neq 0$, let $|v_k| >0$ for some $1 \le k \le n$. Write $v_k$ as $v_k = |v_k| e^{i \phi}$ for some $\phi \in [0, 2 \pi)$. Assume $\mathcal{O}$ is unbounded, then there exist $\alpha \in \mathcal{O}$ with $\alpha > 2\sigma / |v_k|$. Let $\gamma = ({2\sigma}/{|v_k|}) e^{i(-\phi)}$, then $|\gamma| < \alpha$. By the definition of $\mathcal{O}$ \eqref{eq:seto}, we have $Re(\gamma Cv) \in Q^{-1}(Q(0))$. Observe that
$
\|Re(\gamma Cv)\|_1 \ge |Re(\gamma v_k)| = |Re(\frac{2\sigma}{|v_k|} e^{i(-\phi)} |v_k| e^{i \phi})| = |Re(2\sigma)| = 2\sigma.
$
Therefore $Re(\gamma Cv) \notin B_\sigma (0)$, and consequently $Re(\gamma Cv) \notin Q^{-1}(Q(0))$, which draws a contradiction. Therefore $\mathcal{O}$ is bounded.

Next, we define $\beta = \sup \mathcal{O}$. Since $\mathcal{O}$ is non-empty and bounded, we have $0 < \beta < \infty$. Then for any $\epsilon > 0$, there is $\kappa \in \mathbb{C}$ such that 
\begin{equation}
\label{eq:kapniz}
Re(\kappa Cv) \notin Q^{-1}(Q(0)), \ \textrm{and} \ \beta \le |\kappa| < \beta + \epsilon.
\end{equation}
and we will apply this observation to prove Theorem \ref{prop:NC-2-C1} by contradiction.

Assume system \eqref{eq:Sofa-LTI} is finite memory output observable, than there exists an observer $\hat{S}$ \eqref{eq:observer} and $T \in \mathbb{Z}_+$ such that $y_t = \hat{y}_t$ for all $t \ge T$. Consider the input $u_t \equiv 0$, for two initial states $x_0^1$, $x_0^2 \in \mathbb{R}^n$,  we use $y_t^1$, $y_t^2$ to denote the outputs of system \eqref{eq:Sofa-LTI} respectively. Choose $x_0^1 = 0$, then $y_t^1 = Q(0)$ for all $t \in \mathbb{N}$. In \eqref{eq:kapniz}, let $\epsilon = \beta(|\lambda|-1)$, and choose 
$
x_0^2 = Re(\frac{\kappa}{\lambda^T} v).
$
Then for all $0 \le t \le T-1$,
$
CA^t x_0^2 = Re(\frac{\kappa}{\lambda^T} CA^t v) = Re(\frac{\kappa \lambda^t}{\lambda^T} Cv).
$
Since $|{\kappa \lambda^t}/{\lambda^T}| < \beta$, by \eqref{eq:seto}, we see that $CA^t x_0^2 \in Q^{-1}(Q(0))$, and consequently $y_t^2 = Q(0)$ for all $0 \le t \le T-1$. At $t = T$, $CA^T x_0^2 = Re(\kappa Cv) \notin Q^{-1}(Q(0))$, therefore $y_T^2 \neq Q(0)$. 
Now we see that $y_t^1 = y_t^2$ for $0 \le t \le T-1$, and $y_T^1 \neq y_T^2$. Similar to the proof of Theorem \ref{prop:NC-C1}, we can show that $\hat{y}_T^1 = \hat{y}_T^2$, and therefore either $\hat{y}_T^1 \neq \tilde{y}^1_T$ or $\hat{y}_T^2 \neq y_T^2$ or both, and we conclude that system \eqref{eq:Sofa-LTI} is not finite memory output observable.
\end{proof}


\subsection{Algorithmic Verification of the Conditions}
\label{sec:Algo-C1}

A natural question one might ask is how to determine $d(\mathcal{A}, \mathcal{B})$ for sets $\mathcal{A}$ and $\mathcal{B}$ defined in \eqref{eq:set-A} and \eqref{eq:set-B}, respectively.
In this section, we propose an algorithm to compute the distance between sets $\mathcal{A}$ and $\mathcal{B}$ and to determine whether their intersection is empty, 
for the case where matrix $A$ of system \eqref{eq:Sofa-LTI} is Schur-stable ($\rho(A)<1$).

\begin{algorithm}[H]
\caption{Computes $d(\mathcal{A}, \mathcal{B})$, determines if $\mathcal{A} \cap \mathcal{B} =\emptyset$ \label{algo:Compute-C1}}
\begin{algorithmic}[1] 
\Statex {\bf Input}: Matrix $A$, $B$, $C,D$, set $\mathcal{U}$, quantizer $Q$
\State {\bf Compute}:  $h = \max \{\|Bu\|: u \in \mathcal{U}\}$
\State {\bf Compute}:  $s = \sup_{T \ge 0} \sum_{t = 0}^T \|A^t\|$
\State $k \leftarrow 1$.
\Loop
\If {k = 1}
\State {\bf Compute}: $\mathcal{C}_1 = \{Du_1 : u_1 \in \mathcal{U}\} $
\Else 
\State {\bf Compute}:  $\mathcal{C}_k = \{Du_1 + CBu_2 + CABu_3 + \cdots + CA^{k-2}Bu_k : u_1, \dots , u_k \in \mathcal{U}\} $
\EndIf
\State {\bf Compute}: $d_k = \min\{d(y, \mathcal{B}) : y \in \mathcal{C}_k \}$
\If {$d_k > hs \|C\| \|A^{k-1}\|$}
\State {\bf Return}: $\uline{d} = d_k - hs \|C\| \|A^{k-1}\|$
\State {\bf Exit} the loop
\ElsIf {$\mathcal{C}_k \cap \mathcal{B} \neq \varnothing$}
\State {\bf Return}: $ y^* \in \mathcal{C}_k \cap \mathcal{B}$
\State {\bf Exit} the loop
\EndIf
\State $k \leftarrow k+1$.
\EndLoop
\end{algorithmic}
\label{algo}
\end{algorithm}

\begin{lemma}
Given system \eqref{eq:Sofa-LTI}, assume $\rho(A) < 1$ and $0 \in \mathcal{U}$. Consider sets $\mathcal{A}$ and $\mathcal{B}$ defined in \eqref{eq:set-A} and \eqref{eq:set-B} respectively, the following holds:
\renewcommand{\theenumi}{\roman{enumi}}
\begin{enumerate}
\item \label{it:dABg0} $d(\mathcal{A}, \mathcal{B}) > 0$ if and only if Algorithm \ref{algo:Compute-C1} terminates and returns $\uline{d}$, which satisfies $\uline{d} \le d(\mathcal{A}, \mathcal{B})$.
\item $\mathcal{A} \cap \mathcal{B} \neq \varnothing$ if and only if Algorithm \ref{algo:Compute-C1} terminates and returns $y^*$, which satisfies $y^* \in \mathcal{A} \cap \mathcal{B}$.
\end{enumerate}
\end{lemma}

\begin{proof}
First we show item \ref{it:dABg0}). Assume $d(\mathcal{A}, \mathcal{B}) > 0$. Since $\rho(A)<1$,  $\lim_{k \to \infty} \|A^k\| = 0$.
Therefore there is $N \in \mathbb{Z}_+$ such that $hs \|C\| \|A^{k-1}\| < d(\mathcal{A}, \mathcal{B})$ for all $k \ge N$. Recall \eqref{eq:set-A}, we see that $\mathcal{C}_k \subset \mathcal{A}$ for all $k \in \mathbb{Z}_+$. Consequently, $d(\mathcal{C}_k, \mathcal{B}) \ge d(\mathcal{A}, \mathcal{B})$. Since $\mathcal{C}_k$ is a finite set, we see that $d(\mathcal{C}_k, \mathcal{B}) = \min\{d(y, \mathcal{B}) : y \in \mathcal{C}_k \} = d_k$. Therefore, $hs \|C\| \|A^{k-1}\| < d_k$ for all $k \ge N$, and the set $\{k \in \mathbb{Z}_+ : d_k > hs \|C\| \|A^{k-1}\|\}$ is nonempty. Let $k^* = \min\{k \in \mathbb{Z}_+ : d_k > hs \|C\| \|A^{k-1}\|\}$, we see that Algorithm \ref{algo:Compute-C1} terminates at $k = k^*$, and returns $\uline{d} = d_{k^*} - hs \|C\| \|A^{k^*-1}\|$.

Next, we show the other direction of item \ref{it:dABg0}). Assume Algorithm \ref{algo:Compute-C1} terminates at some $k = k^*$, and returns $\uline{d}$. For any $\alpha \in \mathcal{A}$, recall \eqref{eq:set-A}, 
$
\alpha = \sum_{\tau = 0}^{t-1} CA^{t-1-\tau}Bu_\tau + Du_t,  
$
for some  $t \in \mathbb{N}$, and some $u_0, u_1,\dots, u_t \in \mathcal{U}$. Since $0 \in \mathcal{U}$, there are $t' \ge k^*$, and $u_0', u_1',\dots, u_{t'}' \in \mathcal{U}$ such that 
$
\alpha = \sum_{\tau = 0}^{t'-1} CA^{t'-1-\tau}Bu_\tau' + Du_{t'}' = \sum_{\tau = 0}^{t-1} CA^{t-1-\tau}Bu_\tau + Du_t,  
$
where $u_0', u_1',\dots, u_{t'}'$ is either identical to $u_0, u_1,\dots, u_t$, or consists of $u_0, u_1,\dots, u_t$ and a sequence of zero input of appropriate length at the beginning. Since $t' \ge k^*$, we have
$
\alpha = \sum_{\tau = 0}^{t'-k^*} CA^{t'-1-\tau}Bu_\tau' + \sum_{\tau = t'-k^* + 1}^{t'-1} CA^{t'-1-\tau}Bu_\tau' + Du_{t'}'. 
$
We can show that  
$
\|\sum_{\tau = 0}^{t'-k^*} CA^{t'-1-\tau}Bu_\tau' \| \le h s \|C\| \|A^{k^*-1}\|.
$
Let $c = \sum_{\tau = t'-k^* + 1}^{t'-1} CA^{t'-1-\tau}Bu_\tau' + Du_{t'}'$, we note that 
$
c = \sum_{\tau = t'-k^* + 1}^{t'-1} CA^{t'-1-\tau}Bu_\tau' + Du_{t'}' \in \mathcal{C}_{k^*}.
$
Then for any $\beta \in \mathcal{B}$, recall $d_{k^*} = d(\mathcal{C}_{k^*}, \mathcal{B})$
$
\|\alpha - \beta\| \ge \|c - \beta\| - \|\alpha - c\|  \ge d_{k^*} - h s \|C\| \|A^{k^*-1}\|.
$
Since the choices of $\alpha \in \mathcal{A}$ and  $\beta \in \mathcal{B}$ are arbitrary, we see that $d(\mathcal{A}, \mathcal{B}) \ge d_{k^*} - h s \|C\| \|A^{k^*-1}\|$. Since Algorithm \ref{algo:Compute-C1} terminates at $k = k^*$, we see that $\uline{d} = d_{k^*} - h s \|C\| \|A^{k^*-1}\| > 0$, and consequently $d(\mathcal{A}, \mathcal{B}) \ge \uline{d} > 0$. This completes the proof for item \ref{it:dABg0}).

For the second item, assume $\mathcal{A} \cap \mathcal{B} \neq \varnothing$, and let $\alpha_1 \in \mathcal{A} \cap \mathcal{B}$. Recall \eqref{eq:set-A}, we see that $\alpha_1 \in \mathcal{C}_k$ for some $k \in \mathbb{Z}_+$. Therefore, the set $\{k \in \mathbb{Z}_+ : \mathcal{C}_k \cap \mathcal{B} \neq \varnothing \}$ is nonempty. Let $k^* = \min\{k \in \mathbb{Z}_+ : \mathcal{C}_k \cap \mathcal{B} \neq \varnothing \}$, we see that Algorithm \ref{algo:Compute-C1} terminates at $k = k^*$, and returns $y^*$. 
For the backward implication, assume Algorithm \ref{algo:Compute-C1} terminates at some $k \in \mathbb{Z}_+$, and returns $y^*$. Recall \eqref{eq:set-A}, we see that $\mathcal{C}_k \subset \mathcal{A}$. Since $y^* \in \mathcal{C}_k \cap \mathcal{B}$,  we have $y^* \in \mathcal{A} \cap \mathcal{B}$, and consequently $\mathcal{A} \cap \mathcal{B} \neq \varnothing$. This completes the proof of the second item.
\end{proof}

\section{Weak and Asymptotic Output Observability}
\label{sec:WO}

\subsection{Main Results}

Recall that the sufficient conditions for finite memory output observability stated in the previous section are automatically sufficient conditions for weak output observability (Lemma \ref{lem:ClassCimp}). We thus focus on necessary conditions in this section. 

We begin by introducing the concept of {\it unobservable input-output segments}, which will be instrumental in formulating a necessary condition for weak output observability. For any $({\bf u, y}) \in \mathcal{U}^\mathbb{N} \times \mathcal{Y}^\mathbb{N}$ and any $T \in \mathbb{Z_+}$, we use $(\{u_t\}_{t = 0}^T, \{y_t\}_{t = 0}^T)$ to denote the segment of  $({\bf u, y}) $ from time $0$ to time $T$.  

\begin{mydef}
\label{def:NC-C2}
Given a system $P$ \eqref{eq:Sofa} with $|\mathcal{U}| < \infty, 1<|\mathcal{Y}|< \infty$, consider a family $\Psi$ of input and output segments, 
\begin{equation}
\label{eq:IOfampsi}
\Psi = \{\{(\{u_t^{(k,j)}\}_{t = 0}^{T_{(k,j)}}, \{y_t^{(k,j)}\}_{t = 0}^{T_{(k,j)}})\}_{j=1}^{2^k}\}_{k=1}^\infty.
\end{equation}
$\Psi$ is said to be \emph{unobservable} if it satisfies the following items: 
\renewcommand{\theenumi}{\roman{enumi}}
\begin{enumerate}
\item \label{item:IOfamydif1}For any $k \ge 1$, and  any $j \in \{1,2,\dots, 2^{k-1}\}$,
\begin{subequations}
\label{eq:IOfamydif1}
\begin{align}
\label{eq:IOfamyTeq}
& \quad T_{(k,2j-1)}  = T_{(k,2j)}, \\
\label{eq:IOfamyTeqioeq}
& \left\{ \begin{array}{l}
u_t^{(k,2j-1)}  = u_t^{(k,2j)}, \\
 y_t^{(k,2j-1)}  = y_t^{(k,2j)}, 
\end{array} \right.
 0 \le t \le T_{(k,2j-1)} -1,\\
\label{eq:IOfamydif1s}
& \left\{ \begin{array}{l} u_t^{(k,2j-1)}  = u_t^{(k,2j)}, \\
 y_t^{(k,2j-1)}  \neq y_t^{(k,2j)}, 
\end{array} \right.
 t = T_{(k,2j-1)}.
\end{align}
\end{subequations}
\item \label{item:IOfamconn}For any $k \ge 2$, and any $j \in \{1,2,\dots, 2^{k-1}\}$,
\begin{subequations}
\begin{align}
\label{eq:IOfamyTinc}
& \quad T_{(k,2j)}  >  T_{(k-1,j)}, \\
\label{eq:IOfamyuinc}
& \left\{ \begin{array}{l}
u_t^{(k,2j)}  = u_t^{(k-1,j)}, \\
 y_t^{(k,2j)} = y_t^{(k-1,j)}, 
\end{array} \right.
 0 \le t \le T_{(k-1,j)}.
\end{align}
\end{subequations}
\item \label{item:IOcomplete}For any sequence $\{j(k)\}_{k=1}^\infty$ that satisfies $j(k) \in \{1,\dots, 2^k\}$ and $j(k+1) \in \{2j(k)-1, 2j(k)\}$, define $({\bf u, y}) \in \mathcal{U}^\mathbb{N} \times \mathcal{Y}^\mathbb{N}$ as
\begin{equation}
\label{eq:IOcompletedef}
\begin{aligned}
& \left\{ \begin{array}{l}
u_t  = u_t^{(1,j(1))}, \\
 y_t  = y_t^{(1,j(1))}, 
\end{array} \right.  0 \le t  \le T_{(1, j(1))}, \\
& \left\{ \begin{array}{l}
u_t  = u_t^{(k,j(k))}, \\
 y_t  = y_t^{(k,j(k))}, 
\end{array} \right. 
T_{(k-1,j(k-1))} < t  \le T_{(k,j(k))}, k \ge 2,
\end{aligned}
\end{equation}
then $({\bf u, y})$ satisfies
\begin{equation}
\label{eq:IOcomplete}
({\bf u, y}) \in P.
\end{equation}
\end{enumerate}
\hfill $\square$ 
\end{mydef}

\begin{figure}[H]
\centering
\includegraphics[scale = 0.3]{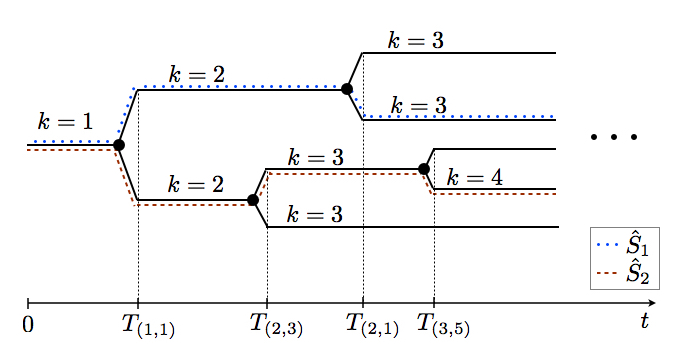} 
\caption{Illustration of $\Psi$. \label{Fig:illus-NC-C2}}
\end{figure}
We use Figure \ref{Fig:illus-NC-C2} to provide some intuition on $\Psi$. The tree structure represents the family $\Psi$ of input and output segments of $P$. Each branch represents two pairs of input and output segments: When the branch bifurcates, it means that the two corresponding outputs are different at this time instant. This exact process is formulated in \eqref{eq:IOfamydif1}. Given an observer $\hat{S}_1$ of $P$, the output of $\hat{S}_1$ will be different from (at least) one of the branches at $t = T_{(1,1)}$, as shown in the figure. Following this branch, the estimation error occurs at one step. Similarly, the output of $\hat{S}_1$ will be different from one of the sub-branches at $t = T_{(2,1)}$, and the estimation error occurs at two steps. Repeating this argument, we see that the estimation error occurs infinitely often for $\hat{S}_1$.  For another observer $\hat{S}_2$, the branches associated with the estimation errors may be different from the ones of $\hat{S}_1$, as shown in Figure \ref{Fig:illus-NC-C2}. Therefore, we propose this tree structure to guarantee that estimation error occurs infinitely often for \emph{any} observer, and consequently the plant $P$ is not weakly output observable. We formalize this result as follows.

\begin{theorem}
\label{prop:NC-C2}
Given a system $P$ \eqref{eq:Sofa} with $|\mathcal{U}| < \infty, 1<|\mathcal{Y}|< \infty$, 
if there is an unobservable family $\Psi$ \eqref{eq:IOfampsi} that satisfies items \ref{item:IOfamydif1}) through \ref{item:IOcomplete}), then system $P$ is not weakly output observable.
\hfill \qedsymbol
\end{theorem}

Essentially, item \ref{item:IOfamydif1}) requires that two output differs at one step but are identical at all previous steps, item \ref{item:IOfamconn}) connects the adjacent stages of $\Psi$, and item \ref{item:IOcomplete}) requires that the branches of $\Psi$ are feasible input and output signals of $P$.  

While the hypotheses in Theorem \ref{prop:NC-C2} seem abstract, one might be interesting in explicitly identifying instances of systems that satisfy these hypotheses. Thus for the specific class of systems \eqref{eq:Sofa-LTI}, we will now present concrete instances of Theorem \ref{prop:NC-C2}. For the purpose of exposition, we consider system \eqref{eq:Sofa-LTI} with $p=1$ and assume the quantizer $Q$ to be right-continuous.

\begin{theorem}
\label{prop:NC-C2-C3}
Consider system \eqref{eq:Sofa-LTI} with $p = 1$,  $0 \in \mathcal{U}$, and $Q$ in the form of \eqref{eq:quat1d}. Assume $\mathcal{B} \cap \mathbb{R}_+ \neq \varnothing$ and let $\beta = \argmin \mathcal{B} \cap \mathbb{R}_+$. If there exist  an eigenvalue-eigenvector pair $(\lambda, v)$ of $A$, a $x^* \in span\{v\}$, and a $u^* \in \mathcal{U}$ such that  $\lambda > 1$, $Cx^* = \beta$, and $A^2 x^* + Bu^* = 0$, then system \eqref{eq:Sofa-LTI} is not weakly  output observable. 
\end{theorem}

\begin{rmk}
Recall Theorem \ref{prop:NC-2-C1}, systems that satisfy the hypotheses in Theorem \ref{prop:NC-C2-C3} are not finite memory output observable. Compared to the hypotheses in Theorem \ref{prop:NC-2-C1}, the requirement  of the existence of $x^*$ and $u^*$ in Theorem \ref{prop:NC-C2-C3} is new.
\end{rmk}

Interestingly, stronger results can be said about the systems that satisfy the hypotheses in Theorem \ref{prop:NC-C2-C3}: In fact, such systems are not even asymptotically observable.

\begin{cor}
Consider a system \eqref{eq:Sofa-LTI} that satisfies the hypotheses in Theorem \ref{prop:NC-C2-C3}, then system \eqref{eq:Sofa-LTI} is not asymptotically   output observable. 
\end{cor}

\subsection{Derivation of Main Results}

First, we make an observation regarding a system $P$ \eqref{eq:Sofa} not being weakly output observable.
\begin{lemma}
\label{lem:C2-neg-eq}
Given a system $P$  \eqref{eq:Sofa} and recall the setup as in Figure \ref{Fig:obs-setup}, $P$ is not weakly output observable if and only if for any observer $\hat{S}$ \eqref{eq:observer}, there is $({\bf u, y}) \in P$ such that 
$y_t \neq \hat{y}_t$ for infinitely many $t \in \mathbb{N}$. 
\end{lemma}

\begin{proof}
By Definition \ref{def:class-C}, $P$ is not weakly output observable if and only if for any  observer $\hat{S}$ \eqref{eq:observer} $\gamma = 0$ is not an observation gain bound of $(P,\hat{S})$. Equivalently, by Definition \ref{eq:obsr-gain}, for any  observer $\hat{S}$ there is $({\bf u, y}) \in P$ such that 
\begin{equation}
\label{eq:C2-neg-eq}
\sup_{T \ge 0} \sum_{t = 0}^T  \| y_t - \hat{y}_t \|  = \infty.
\end{equation}
Equation \eqref{eq:C2-neg-eq} implies $y_t \neq \hat{y}_t$ for infinitely many $t \in \mathbb{N}$. Let $\delta = \min\{\|y_1 - y_2\| : y_1 \neq y_2, y_1, y_2 \in \mathcal{Y}\}$, since $|\mathcal{Y}|<\infty$, the minimum is well-defined and $\delta > 0$. If $y_t \neq \hat{y}_t$ for infinitely many $t \in \mathbb{N}$, then $\| y_t - \hat{y}_t \| > \delta$ for infinitely many $t \in \mathbb{N}$, which implies \eqref{eq:C2-neg-eq}. This completes the proof.
\end{proof}

Next, given an arbitrary observer $\hat{S}$ as in \eqref{eq:observer}, we make an observation regarding its output $\hat{y}_t$ when the segments of $\Psi$ are applied to its input.

\begin{lemma}
\label{lem:famobop}
Given a  system $P$ \eqref{eq:Sofa} and  $\Psi$ \eqref{eq:IOfampsi}, assume the hypotheses in Theorem \ref{prop:NC-C2} are satisfied. For {any} observer $\hat{S}$ \eqref{eq:observer},  define a family $\hat{\Theta}$ of its output segments  as:
\begin{equation}
\label{eq:famoo1}
\hat{\Theta} = \{\{ \{\hat{y}_t^{(k,j)}\}_{t = 0}^{T_{(k,j)}}\}_{j=1}^{2^k}\}_{k=1}^\infty,
\end{equation}
where for all $k \in \mathbb{Z}_+$,  $1 \le j \le 2^k$, $\{\hat{y}_t^{(k,j)}\}_{t = 0}^{T_{(k,j)}}$ is the output of $\hat{S}$ when $u_t$ and $y_t$ in equation \eqref{eq:observer} satisfy
\begin{equation}
\label{eq:famoo2}
\begin{aligned}
u_t & = u_t^{(k,j)}, \ \textrm{for} \ t = 0, 1, \dots, T_{(k,j)}, \\
 y_t & = y_t^{(k,j)}, \ \textrm{for} \ t = 0, 1, \dots, T_{(k,j)} - 1, 
\end{aligned}
\end{equation}
where $u_t^{(k,j)}, y_t^{(k,j)}$ are given by $\Psi$. Then there is a sequence $\{j(k)\}_{k=1}^\infty$ such that for all $k \in \mathbb{Z}_+$, the following are satisfied:
\begin{subequations}
\label{eq:famobop}
\begin{align}
\label{eq:objkex1}
j(k) & \in \{1,\dots,2^k\}, \\
\label{eq:objkex2}
\hat{y}_t^{(k, j(k))} & \neq {y}_t^{(k, j(k))},\ \textrm{for all } t \in \{T_{(i, j(i))}\}_{i=1}^k, \\
\label{eq:objkex3}
u_t^{(k,j(k))} & = u_t^{(k-1, j(k-1))},\ 0 \le t \le T_{(k-1,j(k-1))}, \\
\label{eq:objkex4}
j(k) & \in \{2j(k-1)-1,\  2j(k-1)\}, 
\end{align}
\end{subequations}
where \eqref{eq:objkex3}, \eqref{eq:objkex4} are only required for $k \ge 2$. \hfill \qedsymbol
\end{lemma}

\begin{proof}
We use induction to show this Lemma. For $k=1$, first make an observation of the output $\hat{y}_t$ of the observer. By the dynamics of $\hat{S}$ \eqref{eq:observer}, for any $t \in \mathbb{Z}_+$,
\begin{equation}
\label{eq:ob-deter}
\hat{y}_t = g(f(\dots f(f(q_0, {u}_0, {y}_0), {u}_1, {y}_1)\dots, {u}_{t-1}, {y}_{t-1}), {u}_t).
\end{equation}
Recall \eqref{eq:IOfamydif1}, we have $u_t^{(1,1)} = u_t^{(1,2)}$, for $t = 0,1,\dots,T_{(1,1)}$, and $y_t^{(1,1)} = y_t^{(1,2)}$, for $t = 0,1,\dots,T_{(1,1)}-1$. Let $t = T_{(1,1)}$ in \eqref{eq:ob-deter}, and recall \eqref{eq:famoo2}, we have $\hat{y}_{T_{(1,1)}}^{(1,1)} = \hat{y}_{T_{(1,1)}}^{(1,2)}$. Recall \eqref{eq:IOfamydif1s}, $y_{T_{(1,1)}}^{(1,1)}  \neq y_{T_{(1,1)}}^{(1,2)}$. Consequently, there is $j^* \in \{1,2\}$ such that $\hat{y}_{T_{(1,1)}}^{(1,j^*)} \neq {y}_{T_{(1,1)}}^{(1,j^*)}$. Let $j(1) = j^*$, then $\hat{y}_t^{(1, j(1))}  \neq {y}_t^{(1, j(1))}$, for $ t = T_{(1, j(1))}$, therefore \eqref{eq:famobop} holds at $k=1$.

For $k=2$, recall \eqref{eq:IOfamyTinc}, $T_{(2,2)} > T_{(1,1)}$, and $T_{(2,4)} > T_{(1,2)}$. Recall \eqref{eq:IOfamyTeq}, $T_{(2,j)} > T_{(1,1)}$ for all $j \in \{1,2,3,4\}$. Recall \eqref{eq:IOfamyuinc}, $u_t^{(2,2)}  = u_t^{(1,1)}, u_t^{(2,4)}  = u_t^{(1,2)},  t = 0, 1, \dots, T_{(1,1)}$. Recall \eqref{eq:IOfamyTeqioeq}, \eqref{eq:IOfamydif1s}, 
$u_t^{(2,1)} = u_t^{(2,2)} = u_t^{(1,1)},  u_t^{(2,3)} = u_t^{(2,4)} = u_t^{(1,2)}, \ t = 0, 1, \dots, T_{(1,1)}.$ 
Similarly, by \eqref{eq:IOfamyuinc}, \eqref{eq:IOfamyTeqioeq}, \eqref{eq:IOfamydif1s}, 
$y_t^{(2,1)} = y_t^{(2,2)} = y_t^{(1,1)}, y_t^{(2,3)} = y_t^{(2,4)} = y_t^{(1,2)}, t = 0, 1, \dots, T_{(1,1)}.$
Recall \eqref{eq:ob-deter}, 
\begin{equation}
\label{eq:interstageeqy}
\hat{y}_t^{(2,1)} = \hat{y}_t^{(2,2)} = \hat{y}_t^{(1,1)}, \hat{y}_t^{(2,3)} = \hat{y}_t^{(2,4)} = \hat{y}_t^{(1,2)}, t = T_{(1,1)} . 
\end{equation}
Since \eqref{eq:famobop} holds at $k=1$, $\hat{y}_t^{(1, j(1))}  \neq {y}_t^{(1, j(1))}$, for $ t = T_{(1, j(1))} = T_{(1,1)}$. Recall $j(1) \in \{1,2\}$, by \eqref{eq:interstageeqy}, $\hat{y}_t^{(2,2j(1)-1)} = \hat{y}_t^{(2,2j(1))} = \hat{y}_t^{(1, j(1))}, t = T_{(1,1)}$. Recall \eqref{eq:IOfamyuinc}, ${y}_t^{(2,2j(1)-1)} = {y}_t^{(2,2j(1))} = {y}_t^{(1, j(1))}, t = T_{(1,1)}$. Therefore,
\begin{equation}
\label{eq:IOfamprink2}
\hat{y}_t^{(2,2j(1)-1)} \neq {y}_t^{(2,2j(1)-1)}, \ \hat{y}_t^{(2,2j(1))} \neq {y}_t^{(2,2j(1))}, t = T_{(1,1)}.
\end{equation}
Next, recall \eqref{eq:IOfamydif1}, \eqref{eq:ob-deter}, we have
$\hat{y}_t^{(2,2j(1)-1)} = \hat{y}_t^{(2,2j(1))}, t = T_{(2,2j(1))}.$
Recall \eqref{eq:IOfamydif1s}, 
${y}_t^{(2,2j(1)-1)} \neq {y}_t^{(2,2j(1))}, t = T_{(2,2j(1)-1)} = T_{(2,2j(1))}.$
Consequently, there is $j^* \in \{2j(1)-1,2j(1)\} \subset \{1,2,3,4\}$ such that $\hat{y}_t^{(2,j^*)}  \neq {y}_t^{(2,j^*)}$, $t = T_{(2,j^*)}$. Let $j(2) = j^*$, then \eqref{eq:objkex1}, \eqref{eq:objkex4} hold. By \eqref{eq:IOfamprink2}, and $\hat{y}_t^{(2,j^*)}  \neq {y}_t^{(2,j^*)}$, $t = T_{(2,j^*)}$, \eqref{eq:objkex2} holds. By \eqref{eq:IOfamydif1}, \eqref{eq:IOfamyuinc}, we see \eqref{eq:objkex3} holds. Therefore  \eqref{eq:famobop} holds at $k=2$.

Assume that \eqref{eq:famobop} holds for some $k \ge 2$. Recall \eqref{eq:IOfamyTeqioeq}, \eqref{eq:IOfamydif1s}, \eqref{eq:IOfamyuinc}, $u_t^{(k+1, 2j(k)-1)} = u_t^{(k+1, 2j(k))} = u_t^{(k,j(k))}$, and $y_t^{(k+1, 2j(k)-1)} = y_t^{(k+1, 2j(k))} = y_t^{(k,j(k))}$, $t = 0,1,\dots, T_{(k,j(k))}$. Recall \eqref{eq:famoo2}, \eqref{eq:ob-deter}, we see that 
$
\hat{y}_t^{(k+1, 2j(k)-1)} = \hat{y}_t^{(k+1, 2j(k))} = \hat{y}_t^{(k,j(k))}, t = 0,1,\dots, T_{(k,j(k))}.
$
Since \eqref{eq:objkex2} holds for $k$, we see that 
\begin{equation}
\label{eq:IOfamprinkk}
\left\{ \begin{array}{l}
\hat{y}_t^{(k+1, 2j(k)-1)} \neq y_t^{(k+1, 2j(k)-1)}, \\
  \hat{y}_t^{(k+1, 2j(k))} \neq y_t^{(k+1, 2j(k))}, 
  \end{array}\right.
  \forall \ t \in \{T_{(i, j(i))}\}_{i=1}^k.
\end{equation}

At $t = T_{(k+1, 2j(k))}$, by \eqref{eq:IOfamydif1}, \eqref{eq:ob-deter}, 
$
\hat{y}_t^{(k+1, 2j(k)-1)} = \hat{y}_t^{(k+1, 2j(k))}, t = T_{(k+1, 2j(k))}. 
$
By \eqref{eq:IOfamyTeq}, \eqref{eq:IOfamydif1s}, ${y}_t^{(k+1, 2j(k)-1)} \neq {y}_t^{(k+1, 2j(k))}, t = T_{(k+1, 2j(k))}.$ Therefore there is $j^* \in \{2j(k)-1, 2j(k)\}$ such that $\hat{y}_t^{(k+1, j^*)} \neq {y}_t^{(k+1, j^*)}$ at $t =  T_{(k+1, 2j(k))} = T_{(k+1, j^*)}$. Let $j(k+1) = j^*$, and recall \eqref{eq:IOfamprinkk}, we see that \eqref{eq:objkex2} holds for $k+1$. Since $j(k+1) = j^* \in \{2j(k)-1, 2j(k)\}$, and \eqref{eq:objkex1} holds for $k$, we see that   \eqref{eq:objkex1}, \eqref{eq:objkex4} holds for $k+1$. By \eqref{eq:IOfamydif1}, \eqref{eq:IOfamyuinc}, we see \eqref{eq:objkex3} holds for $k+1$. We see that  \eqref{eq:famobop} holds for $k+1$. This completes the derivation of the existence of $\{j(k)\}_{k=1}^\infty$ such that \eqref{eq:famobop} holds for all $k \in \mathbb{Z}_+$.
\end{proof}

Based on the observation made in Lemma \ref{lem:famobop}, given $\hat{S}$ we construct an input-output pair of $P$ such that  the estimation error occurs infinitely many often.

\begin{lemma}
Given a  system $P$ \eqref{eq:Sofa} and  $\Psi$ \eqref{eq:IOfampsi}, assume the hypotheses in Theorem \ref{prop:NC-C2} are satisfied. Given an observer $\hat{S}$ \eqref{eq:observer}, let $\{j(k)\}_{k=1}^\infty$ be such that \eqref{eq:famobop} holds, define an input-output pair $({\bf u, y}) \in \mathcal{U}^\mathbb{N} \times \mathcal{Y}^\mathbb{N}$ as 
\begin{equation}
\label{eq:ioinfecon}
\begin{aligned}
& \left\{ \begin{array}{l}
u_t  = u_t^{(1,j(1))}, \\
 y_t  = y_t^{(1,j(1))}, 
\end{array} \right.  0 \le t  \le T_{(1, j(1))}, \\
& \left\{ \begin{array}{l}
u_t  = u_t^{(k,j(k))}, \\
 y_t  = y_t^{(k,j(k))}, 
\end{array} \right. 
T_{(k-1,j(k-1))} < t  \le T_{(k,j(k))}, k \ge 2,
\end{aligned}\end{equation}
then $({\bf u, y})$ is well-defined and $({\bf u, y}) \in P$. \hfill \qedsymbol
\end{lemma}

\begin{proof}
First we show that $({\bf u, y})$ \eqref{eq:ioinfecon} is well-defined. Let $L(k) = \{T_{(k-1,j(k-1))} + 1, \dots, T_{(k,j(k))}\} \subset \mathbb{Z}_+$ for all $k \ge 2$. We observe that 
\begin{subequations}
\begin{align}
\label{eq:ioinfecont}
\{0,1,\dots,T_{(1,j(1))}\} & \cup (\bigcup_{k=2}^\infty L(k)) = \mathbb{N}, \\
\label{eq:ioinfecontemp}
L(k_1) & \cap L(k_2) = \varnothing, \ \textrm{if} \ k_1 \neq k_2.
\end{align}
\end{subequations} 

To see this, by \eqref{eq:objkex4}, \eqref{eq:IOfamyTeq}, \eqref{eq:IOfamyTinc}, 
\begin{equation}
\label{eq:ioinfecontinc}
T_{(k,j(k))} > T_{(k-1,j(k-1))},
\end{equation}
 or equivalently 
$T_{(k,j(k))} \ge T_{(k-1,j(k-1))} +1$, for all $k \ge 2$. Consequently, $T_{(k,j(k))} \ge T_{(1,j(1))} +(k-1)$, for all $k \ge 2$, and 
$\sup_{k \in \mathbb{Z}_+} T_{(k,j(k))} = \infty.
$
For any $t \in \mathbb{N}$, if $t > T_{(1,1)}$, let $k(t) = \min\{k \in \mathbb{Z}_+ : T_{(k,j(k))} \ge t \}$. Since $\{k \in \mathbb{Z}_+ : T_{(k,j(k))} \ge t \}$ is none-empty, $k(t)$ is well-defined (pp. 28, \cite{Nicholson}), and $t \in L(k(t))$. Therefore \eqref{eq:ioinfecont} holds.
For any $k_1, k_2 \ge 2$, and $k_1 \neq k_2$, without loss of generality, let $k_1 < k_2$. Assume $t^* \in L(k_1) \cap L(k_2)$, then $t^* > T_{(k_2 - 1,j(k_2 - 1))}$, and $t^* \le T_{(k_1,j(k_1))}$. Since $k_2 - 1 \ge k_1$, by \eqref{eq:ioinfecontinc}, $T_{(k_2 - 1,j(k_2 - 1))} \ge T_{(k_1,j(k_1))}$, but $T_{(k_2 - 1,j(k_2 - 1))} < t^* \le T_{(k_1,j(k_1))}$, which draws a contradiction. Therefore \eqref{eq:ioinfecontemp} holds. Consequently, $({\bf u, y})$ \eqref{eq:ioinfecon} is well-defined.

By \eqref{eq:objkex1}, \eqref{eq:objkex4}, and item \ref{item:IOcomplete}) of $\Psi$ (see \eqref{eq:IOcomplete}), $({\bf u, y})$ \eqref{eq:ioinfecon} satisfies $({\bf u, y}) \in P$. 
\end{proof}

Lastly, we apply the previous observations to show Theorem \ref{prop:NC-C2}.
\begin{proof} (Theorem \ref{prop:NC-C2})
For any observer $\hat{S}$ as in \eqref{eq:observer}, consider $({\bf u, y}) \in P$  defined as in \eqref{eq:ioinfecon}. In the setup in Figure \ref{Fig:obs-setup}, let $\{\hat{y}_t\}_{t=0}^\infty$ be the output of $\hat{S}$ corresponding with $({\bf u, y})$. Recall \eqref{eq:ioinfecon}, we see that $u_t  = u_t^{(1,j(1))}$,  $0 \le t  \le T_{(1, j(1))}$. Recall \eqref{eq:famobop}, we see that for any $k \in \mathbb{Z}_+$, we have $u_t  = u_t^{(k,j(k))}$,  $0 \le t  \le T_{(1, j(1))}$. Similarly, we can show that for any $k \ge 2$, $u_t  = u_t^{(k,j(k))}$,  $T_{(1, j(1))} + 1 \le t  \le T_{(2, j(2))}$. Repeat this argument, then for any $k \in \mathbb{Z}_+$, we have $u_t  = u_t^{(k,j(k))}$,  $0 \le t  \le T_{(k, j(k))}$. Similarly, we can show that for any $k \in \mathbb{Z}_+$, $y_t  = y_t^{(k,j(k))}$,  $0 \le t  \le T_{(k, j(k))} - 1$. Recall the definition of $\hat{\Theta}$ (see \eqref{eq:famoo2}), and \eqref{eq:ob-deter}, we have 
$
\hat{y}_t = \hat{y}_t^{(k,j(k))}, t = T_{(k, j(k))}, \ \textrm{for all } \ k \in \mathbb{Z}_+.
$
Recall \eqref{eq:ioinfecon}, ${y}_t = {y}_t^{(k,j(k))}, t = T_{(k, j(k))}, \ \textrm{for all }  k \in \mathbb{Z}_+.$ Recall \eqref{eq:objkex2},  $\hat{y}_t^{(k,j(k))} \neq {y}_t^{(k,j(k))}, t = T_{(k, j(k))}, \ \textrm{for all }  k \in \mathbb{Z}_+.$ Therefore,  
\begin{equation}
\label{eq:IOfaminfercon}
\hat{y}_t \neq y_t, \quad t = T_{(k, j(k))}, \forall \ k \in \mathbb{Z}_+. 
\end{equation}
Recall \eqref{eq:ioinfecontinc}, we see that $\hat{y}_t \neq y_t$ for infinitely many $t \in \mathbb{N}$. By  Lemma \ref{lem:C2-neg-eq}, $P$ is not weakly output observable.
\end{proof}

We first construct a family $\Psi$ of input-output segments  of system \eqref{eq:Sofa-LTI}, and then show that the constructed $\Psi$ satisfies the hypotheses in Theorem \ref{prop:NC-C2}. 
Within the scope of this derivation, we use $``s"$ to denote the initial state $x_0$ of system \eqref{eq:Sofa-LTI}. 

Consider a system \eqref{eq:Sofa-LTI} that satisfies the hypotheses in Theorem \ref{prop:NC-C2-C3} and let $T \in \mathbb{Z}_+$  and $T \ge 2$, we define a quantity $s_o$ as
\begin{equation}
s_o = \frac{x^*}{\lambda^T}.
\end{equation}
Next, define a family $\mathcal{S}$ of initial states of system \eqref{eq:Sofa-LTI}  as
\begin{equation}
\label{eq:faminits}
\mathcal{S} = \{\{s_{(k,j)}\}_{j=1}^{2^k}\}_{k=1}^\infty,
\end{equation}
where 
\begin{equation}
\label{eq:faminits1}
s_{(1,1)} = 0, \quad s_{(1,2)} = s_o,
\end{equation}
and for all $k \ge 2$, all $j = 1,2,\dots, 2^{k-1}$,
\begin{equation}
\label{eq:faminitsinc}
s_{(k, 2j-1)} = s_{(k - 1, j)}, \quad s_{(k, 2j)} = s_{(k - 1, j)} + (\frac{1}{\lambda^T})^{k-1} s_o.   
\end{equation}
Then $s_{(k,j)}$ is defined for all $k \in \mathbb{Z}_+$ and $j \in \{1,\dots, 2^k\}$.

At the same time, define a family $\mathcal{I}$ of input segments of system \eqref{eq:Sofa-LTI} as 
\begin{equation}
\label{eq:famii}
\mathcal{I} = \{\{\{u_t^{(k,j)}\}_{t = 0}^{T_{(k,j)}}\}_{j=1}^{2^k}\}_{k=1}^\infty,
\end{equation}
where for all $k \in \mathbb{Z}_+$, all $j = 1,2,\dots, 2^{k}$,
\begin{equation}
\label{eq:famiiT}
T_{(k,j)} = k\cdot T,
\end{equation}
and 
\begin{equation}
\label{eq:famiii}
u_t^{(1,1)} = u_t^{(1,2)} = 0, \ t = 0, 1, \dots, T,
\end{equation}
and for all $k \ge 2$, all $j = 1,2,\dots, 2^{k-2}$,
\begin{equation}
\label{eq:famiiinc}
\begin{aligned}
&\left\{ \begin{array}{l}
u_t^{(k,4j)} = u_t^{(k,4j-1)} =u_t^{(k-1,2j)}, \\
 u_t^{(k,4j-2)} = u_t^{(k,4j-3)} = u_t^{(k-1,2j-1)}, 
\end{array} \right.
 0 \le t \le (k-1) T, \\
&\left\{ \begin{array}{l}
u_t^{(k,4j)} = u_t^{(k,4j-1)} = u^*, \\
 u_t^{(k,4j-2)} = u_t^{(k,4j-3)} = 0, 
\end{array} \right. 
\quad \quad \quad t = (k-1)T + 1, \\
&\quad u_t^{(k,4j-i)}  = 0,  \quad 0 \le i \le 3, \quad (k-1)T + 2 \le t \le k T, \\
\end{aligned}
\end{equation}
then $u_t^{(k,j)}$ is defined for all $k \in \mathbb{Z}_+$, $j \in \{1,\dots, 2^k\}$, and $t \in \{0,\dots, kT\}$.

Given $\mathcal{S}$ and $\mathcal{I}$ defined in the preceding,  define a family ${{\Theta}}$ of output segments of system \eqref{eq:Sofa-LTI} as:
\begin{equation}
\label{eq:famoy}
{{\Theta}} = \{\{ \{{{y}}_t^{(k,j)}\}_{t = 0}^{T_{(k,j)}}\}_{j=1}^{2^k}\}_{k=1}^\infty,
\end{equation}
where for all $k \in \mathbb{Z}_+$,  $1 \le j \le 2^k$, $\{{{y}}_t^{(k,j)}\}_{t = 0}^{T_{(k,j)}}$ is the quantized output $y_t$  \eqref{eq:Sofa-LTI-qo}  of system \eqref{eq:Sofa-LTI}, when $u_t$ and $x_t$ in equation \eqref{eq:Sofa-LTI} satisfy
\begin{equation}
\label{eq:famoy1}
\begin{aligned}
u_t & = u_t^{(k,j)}, \ \textrm{for} \ t = 0, 1, \dots, T_{(k,j)}, \\
x_t & = s_{(k,j)}, \ \textrm{for} \ t = 0. 
\end{aligned}
\end{equation}
Essentially, ${{y}}_t^{(k,j)}$ is the quantized output of system \eqref{eq:Sofa-LTI} when $u_t^{(k,j)}$ is applied to its input, and its initial state is $s_{(k,j)}$. In the following, we also use $x_t^{(k,j)}$ to denote the state $x_t$ of system \eqref{eq:Sofa-LTI} corresponding with \eqref{eq:famoy1}. 

Given $\mathcal{I}$  and ${\Theta}$, define $\Psi$ as 
\begin{equation}
\label{eq:famio}
\Psi = \{\{(\{u_t^{(k,j)}\}_{t = 0}^{T_{(k,j)}}, \{y_t^{(k,j)}\}_{t = 0}^{T_{(k,j)}})\}_{j=1}^{2^k}\}_{k=1}^\infty,
\end{equation}
where $u_t^{(k,j)}$, $y_t^{(k,j)}$ are given by $\mathcal{I}$ \eqref{eq:famii} and ${\Theta}$ \eqref{eq:famoy} respectively. 

In the following, we will show that the $\Psi$ defined as in \eqref{eq:famio} satisfies items \ref{item:IOfamydif1}) to \ref{item:IOcomplete}) in Theorem \ref{prop:NC-C2} and therefore the system is not weakly output observable. Toward this end, we start with making an observation about $\mathcal{S}$.  

\begin{lemma}
Given a system \eqref{eq:Sofa-LTI} that satisfies the hypotheses in Theorem \ref{prop:NC-C2-C3} and an integer  $T \ge 2$, let $q = \lambda^{-T}$. Then $\mathcal{S}$ defined  in \eqref{eq:faminits} is such that for any $k \in \mathbb{Z}_+$ and $j \in \{1,\dots, 2^k\}$, 
\begin{equation}
\label{eq:skjexplic}
s_{(k,j)} = (\alpha_1^{(k,j)}  + \alpha_2^{(k,j)} \cdot q  + \cdots + \alpha_k^{(k,j)} \cdot q^{k-1}) \cdot s_o, 
\end{equation}
where  $\alpha_l^{(k,j)}$, $l \in \{1,\dots, k\}$, is defined as
\begin{equation}
\label{eq:alphakj}
\alpha_l^{(k,j)} = \left\{\begin{array}{l}
0, \ \textrm{if} \ 0 \le (j-1) \hspace{-4pt}\mod (2^{k-l+1}) < \frac{2^{k-l+1}}{2}, \\
1, \ \textrm{if} \ \frac{2^{k-l+1}}{2} \le (j-1) \hspace{-4pt}\mod (2^{k-l+1}) < 2^{k-l+1}. \\
\end{array}\right.
\end{equation}
\end{lemma}

\begin{proof}
We use induction to show this result. For $k=1$, by \eqref{eq:alphakj}, $\alpha_1^{(1,1)} = 0$, $\alpha_1^{(1,2)} = 1$. Recall \eqref{eq:faminits1}, we see that \eqref{eq:skjexplic} hold for $k=1$. 

Next, assume \eqref{eq:skjexplic} hold for some $k \ge 1$, about $\alpha_l^{(k,j)}$ \eqref{eq:alphakj}, observe that for any $l \in \{1,\dots,k\}$
\begin{equation}
\label{eq:alphakjprop}
\alpha_l^{(k,j)} = \alpha_l^{(k+1,2j-1)} = \alpha_l^{(k+1,2j)}.
\end{equation}
To see this, let $(j-1) \hspace{-4pt}\mod 2^{k-l+1} = b$, for some $0 \le b < 2^{k-l+1}$. Then $j-1 = 2^{k-l+1} \cdot a + b$, for some unique $a \in \mathbb{N}$ (pp. 32, \cite{Nicholson}).  Therefore $((2j-1)-1) \hspace{-4pt}\mod 2^{(k+1)-l+1} = 2b$, and $(2j-1) \hspace{-4pt}\mod 2^{(k+1)-l+1} = 2b+1$. If $\alpha_l^{(k,j)} = 0$, recall \eqref{eq:alphakj}, $ \alpha_l^{(k+1,2j-1)} = \alpha_l^{(k+1,2j)} = 0$. Similarly, we can show that if $\alpha_l^{(k,j)} = 1$, then $ \alpha_l^{(k+1,2j-1)} = \alpha_l^{(k+1,2j)} = 1$. Therefore \eqref{eq:alphakjprop} holds. By \eqref{eq:faminitsinc}, \eqref{eq:alphakjprop}, observe that for all $j = 1,2,\dots, 2^k$, $\alpha_{k+1}^{(k+1,2j-1)} = 0$, and $\alpha_{k+1}^{(k+1,2j)} = 1$, we see that for any $1 \le j \le 2^k$
$
s_{(k+1, 2j-1)}  =  (\alpha_1^{(k,j)}  + \alpha_2^{(k,j)} \cdot q  + \cdots + \alpha_k^{(k,j)} \cdot q^{k-1} + \alpha_{k+1}^{(k+1,2j-1)} q^k) \cdot s_o,
$
$
s_{(k+1, 2j)} =  (\alpha_1^{(k,j)}  + \alpha_2^{(k,j)} \cdot q  + \cdots + \alpha_k^{(k,j)} \cdot q^{k-1} + \alpha_{k+1}^{(k+1,2j-1)} \cdot q^k) \cdot s_o.
$
Therefore \eqref{eq:skjexplic} hold for  $k + 1$. By induction, \eqref{eq:skjexplic} hold for all $k \in \mathbb{Z}_+$.
\end{proof}

We also make an observation about $\mathcal{I}$ \eqref{eq:famii} in the following.
\begin{lemma} 
Given a system \eqref{eq:Sofa-LTI} that satisfies the hypotheses in Theorem \ref{prop:NC-C2-C3} and an integer  $T \ge 2$, $\mathcal{I}$ defined in \eqref{eq:famii} is such that for any $k \ge 2$, $j \in \{1,2,\dots, 2^k\}$, 
\begin{subequations}
\label{eq:utkjexpl}
\begin{align}
\label{eq:utkjexplaz1}
u_t^{(k,j)} &= 0, \quad t = 0, 1, \dots, T, \\
\label{eq:utkjexpltp1}
u_t^{(k,j)} &= \alpha_1^{(k,j)} u^*, \quad t = T + 1, \\
\label{eq:utkjexplss}
u_t^{(k,j)} &= u_{t-T}^{(k-1,h(k,j))}, \quad t = T + 2, T + 3, \dots, kT.
\end{align}
\end{subequations}
where  function $h(k,j)$ is defined as
\begin{equation}
\label{eq:funchkj}
h(k,j) = (j-1) \hspace{-4pt}\mod 2^{k-1} + 1, \forall \ k \in \mathbb{Z}_+,  j \in \mathbb{Z}_+.
\end{equation}
\end{lemma}

\begin{proof}
We use induction to show \eqref{eq:utkjexpl} holds. For $k=2$, recall  $u_t^{(2,1)} = u_t^{(2,2)} = 0$ for $0 \le t \le 2T$,  $u_t^{(2,3)} = u_t^{(2,4)} = 0$ for $0 \le t \le 2T$ and $t \neq T+1$, $u_t^{(2,3)} = u_t^{(2,4)} = u^*$ for $t = T+1$, and $u_t^{(1,1)} = u_t^{(1,2)} = 0,$ for $0 \le t \le T$. By \eqref{eq:alphakj}, $\alpha_1^{(2,1)} = \alpha_1^{(2,2)} = 0$, $\alpha_1^{(2,3)} = \alpha_1^{(2,4)} = 1$. By \eqref{eq:funchkj}, $h(2,1) = 1, h(2,2) = 2, h(2,3) = 1, h(2,4) = 2$.  We see that \eqref{eq:utkjexpl} holds for $k=2$.

Assume \eqref{eq:utkjexpl} holds for some $k\ge2$. Recall \eqref{eq:famiiinc}, for all $j \in \{1,\dots,2^{k}\}$
\begin{equation}
\label{eq:famiiincob}
u_t^{(k+1,2j - 1)} = u_t^{(k+1,2j)} = u_t^{(k,j)}, t = 0,1,\dots,kT.
\end{equation}
By assumption, $u_t^{(k+1,2j - 1)} = u_t^{(k+1,2j)} = u_t^{(k,j)} = 0$, $0 \le t \le T$, $1 \le j \le  2^k$, and \eqref{eq:utkjexplaz1} holds for $k+1$. At $t = T+1$, by \eqref{eq:famiiincob} and \eqref{eq:utkjexpltp1}, $u_t^{(k+1,2j - 1)} = u_t^{(k+1,2j)} = u_t^{(k,j)} = \alpha_1^{(k,j)} u^*$. Recall \eqref{eq:alphakjprop}, $\alpha_1^{(k,j)} = \alpha_1^{(k+1,2j-1)} = \alpha_1^{(k+1,2j)}.$ Therefore $u_t^{(k+1,2j - 1)} = \alpha_1^{(k+1,2j-1)} u^*$, and $u_t^{(k+1,2j )} = \alpha_1^{(k+1,2j)} u^*$ for $t = T+1$. Consequently, \eqref{eq:utkjexpltp1} holds for $k+1$.

Next, we show \eqref{eq:utkjexplss} holds for $k+1$. Recall \eqref{eq:funchkj}, we can show that for all $k \ge 2$,  $j \in \{1,2,\dots, 2^{k-1}\}$,
\begin{equation}
\label{eq:funchkjprop}
h(k,2j-1) = 2 h (k-1, j) - 1, \quad h(k,2j) = 2 h (k-1, j).
\end{equation} 

Since \eqref{eq:utkjexplss} holds for $k$ by assumption, recall \eqref{eq:famiiincob}, we have $u_t^{(k+1,2j-1)}= u_t^{(k+1,2j)}=u_t^{(k,j)} = u_{t-T}^{(k-1,h(k,j))}, T+2 \le t \le kT, 1 \le j \le 2^k.$ Recall \eqref{eq:famiiincob}, we have $ u_{t-T}^{(k-1,h(k,j))}  = u_{t-T}^{(k, 2h(k,j))} $, $T+2 \le t \le kT$. Combine the preceding and recall \eqref{eq:funchkjprop}, we see that for  $j \in \{1,\dots,2^{k+1}\}$,
\begin{equation}
\label{eq:ufam851}
u_t^{(k+1,j)} = u_{t-T}^{(k, h(k+1,j))}, \quad T+2 \le t \le kT.
\end{equation}
Recall \eqref{eq:famiiinc}, we see that for $j \in \{1,\dots,2^{k+1}\}$, 
\begin{equation}
\label{eq:ufam852}
u_t^{(k+1,j)} = 0 = u_{t-T}^{(k, h(k+1,j))}, \quad kT+2 \le t \le (k+1)T.
\end{equation}
Recall \eqref{eq:famiiinc}, we observe that at $t = (k-1)T + 1$, $u_t^{(k,j)}$ is determined by $j\hspace{-4pt}\mod 4$. Also observe that for  $k\ge2$, $1 \le j \le 2^{k+1}$, $j \hspace{-4pt}\mod 4  = h(k+1, j) \hspace{-4pt}\mod 4$. Therefore, $u_{kT + 1}^{(k+1,j)} = u_{(k-1)T + 1}^{(k, h(k+1, j))}$. Consequently,
\begin{equation}
\label{eq:ufam853}
u_t^{(k+1,j)} = u_{t-T}^{(k, h(k+1,j))},  \quad t = kT + 1.
\end{equation}
By \eqref{eq:ufam851}, \eqref{eq:ufam852}, \eqref{eq:ufam853}, we see that \eqref{eq:utkjexplss} holds for $k+1$, and \eqref{eq:utkjexpl} holds for all $k \ge 2$.
\end{proof}

Now we proceed to make observations about $\Psi$. 

\begin{lemma}
\label{lem:Psisfyit12}
Given a system \eqref{eq:Sofa-LTI} that satisfies the hypotheses in Theorem \ref{prop:NC-C2-C3}, 
there is a $T \in \mathbb{Z}_+$ such that $\Psi$ defined   in \eqref{eq:famio} satisfies items \ref{item:IOfamydif1}) and \ref{item:IOfamconn}) in Theorem \ref{prop:NC-C2}. 
\end{lemma}

\begin{proof}
We use induction to show this result. First let $T \ge 2$. 
 For $k=1$, recall \eqref{eq:famiiT}, $T_{(1,1)} = T_{(1,2)} = T$. Recall \eqref{eq:famiii}, $u_t^{(1,1)} = u_t^{(1,2)} = 0,  0 \le t \le T$. By the definition of $\Theta$ \eqref{eq:famoy1} and $s_{(1,1)} = 0$, $y_t^{(1,1)} = Q(0)$ for $0 \le t \le T$. Recall $s_{(1,2)} = s_o \in span\{v\}$, $\lambda > 1$, $\beta = \argmin\{|b| : b \in \mathcal{B}, b>0\}$,  $y_t^{(1,2)} = Q(0)$ for $0 \le t \le T-1$. At $t = T$, $y_t^{(1,2)} =  Q(\beta) \neq Q(0)$. Therefore, $y_t^{(1,1)} = y_t^{(1,2)}$, $0 \le t \le T-1$, and $y_t^{(1,1)} \neq y_t^{(1,2)}$, $t = T$, and item \ref{item:IOfamydif1}) is satisfied for $k=1$.

For $k=2$,  recall \eqref{eq:famiiT}, for any $1\le j \le 4$, $T_{(2,j)} = 2T$. Recall \eqref{eq:faminits1}, \eqref{eq:faminitsinc}, $s_{(2,1)} = 0,  s_{(2,2)} = \lambda^{-T} s_o, s_{(2,3)} = s_o, s_{(2,4)} = s_o + \lambda^{-T} s_o$. By \eqref{eq:famiii}, \eqref{eq:famiiinc}, $u_t^{(2,1)} = u_t^{(2,2)} = 0$ for $0 \le t \le 2T$, and $u_t^{(2,3)} = u_t^{(2,4)} = 0$ for $0 \le t \le 2T$ and $t \neq T+1$, $u_t^{(2,3)} = u_t^{(2,4)} = u^*$ for $t = T+1$. Therefore $y_t^{(2,1)} = Q(0)$ for $0 \le t \le 2T$, and $y_t^{(2,2)} = Q(CA^t (\lambda^{-T} s_o)) = Q(0)$ for $0 \le t \le 2T - 1$. At $t = 2T$, $y_t^{(2,2)} = Q(\beta) \neq Q(0)$. Consequently, items \ref{item:IOfamydif1}) and \ref{item:IOfamconn}) are satisfied when $k=2, j=1$. 

For the case $k=2, j=2$, by the definition of $\Theta$ \eqref{eq:famoy1}, we can  show that
\begin{equation}
\label{eq:famioeplfyt23}
y_t^{(2,3)} = \left\{\begin{array}{ll}
Q(0), & t = 0, 1, \dots, T - 1, \\
Q(\beta), & t = T, \\
Q(\lambda \beta + Du^*), & t = T+1, \\
Q(0), & t = T+2, \dots, 2T.
\end{array}\right.
\end{equation}

For $y_t^{(2,4)}$, at $t = T-1$,  $y_t^{(2,4)} =  Q(\lambda^{-1}(1 + \lambda^{-T}) \beta)$. Choose $T \in \mathbb{Z}_+$ such that 
\begin{equation}
\label{eq:ncc2cT1}
\lambda^T > \frac{\lambda}{\lambda-1}.
\end{equation}
Since $\lambda>1$, such a choice of $T$ always exist, for example let $T > log_\lambda(\frac{\lambda}{\lambda-1})$. Then we have  $y_t^{(2,4)} = Q(0)$, $t = T-1$. 
For any $0 \le t \le T-2$, we see that $0 < CA^{t} (s_o + \lambda^{-T} s_o) < \lambda^{-1} \beta (1 + \lambda^{-T})$, and consequently $y_t^{(2,4)} = Q(0)$. At $t = T$, $y_t^{(2,4)}  = Q((1 + \lambda^{-T}) \beta)$. Recall $Q$ \eqref{eq:quat1d}, there is $\delta_1>0$ such that 
$
Q(y) = Q(\beta),  \forall \ y \in [\beta, \beta + \delta_1).
$
Choose $T \in \mathbb{Z}_+$ such that 
\begin{equation}
\label{eq:ncc2cT2}
 \frac{1}{1 - \lambda^{-T}} \beta < \beta + \delta_1.
\end{equation}
Since $\lambda > 1$, \eqref{eq:ncc2cT2} is satisfied for all $T$ sufficiently large. Then $y_t^{(2,4)} = Q(\beta)$,  $t = T$. At $t = T + 1$, $y_t^{(2,4)} = Q((1 + \lambda^{-T}) \lambda\beta + Du^*)$.
Similarly, recall  \eqref{eq:quat1d}, there is $\delta_2>0$ such that 
$
Q(y) = Q(\lambda\beta + Du^*),  \forall \ y \in [\lambda\beta + Du^*, \lambda\beta + Du^* + \delta_2).
$
Choose $T \in \mathbb{Z}_+$ such that 
\begin{equation}
\label{eq:ncc2cT3}
 \frac{1}{1 - \lambda^{-T}} \lambda\beta < \lambda\beta + \delta_2.
\end{equation}
Then $y_t^{(2,4)} = Q(\lambda\beta + Du^*)$, $t = T+1$.
At $t = T+2$, the system state $x_t^{(2,4)} = A^2 x^* + Bu^* + \lambda^2 s_o$. By assumption $A^2 x^* + Bu^* =  0$, we have $x_{T+2}^{(2,4)} = \lambda^2 s_o = A^2 s_o$.
Recall $u_t^{(2,4)} = 0$ for $T+2 \le t \le 2T$,  we can show that
\begin{equation*}
y_t^{(2,4)} = \left\{\begin{array}{ll}
Q(0), & t = 0, 1, \dots, T - 1, \\
Q(\beta), & t = T, \\
Q(\lambda\beta + Du^*), & t = T+1, \\
Q(0), & t = T+2, \dots, 2T - 1, \\
Q(\beta), & t = 2T.
\end{array}\right.
\end{equation*}
Recall \eqref{eq:famioeplfyt23}, and the explicit form of $y_t^{(1,2)}$, we see that  items \ref{item:IOfamydif1}) and \ref{item:IOfamconn}) are satisfied when $k=2, j=2$. 
We conclude that $\Psi$ \eqref{eq:famio} satisfies items \ref{item:IOfamydif1}) and \ref{item:IOfamconn}) in Theorem \ref{prop:NC-C2} when $k=2$.

Next, assume items \ref{item:IOfamydif1}) and \ref{item:IOfamconn}) are satisfied for some $k \ge 2$, we will show that they are satisfied for $k+1$.  

Recall \eqref{eq:famiiinc}, we see that 
\begin{equation}
\label{eq:ustukp1d2j}
u_t^{(k+1,2j-1)} = u_t^{(k+1,2j)}, 0 \le t \le (k+1)T,
\end{equation}
for all $1\le j \le 2^k$. Note that $(1 + q + \cdots + q^k)  < \frac{1}{1-q} $, and  recall \eqref{eq:skjexplic}, \eqref{eq:alphakj},  \eqref{eq:utkjexplaz1}, for all $0\le t \le T- 1$, $C x_t^{(k+1,2j-1)} < \lambda^{-1} \frac{1}{1-\lambda^{-T}} \beta$. Recall \eqref{eq:ncc2cT1}, $\lambda^{-1} \frac{1}{1 - \lambda^{-T}} < 1,$ and $y_t^{(k+1,2j-1)}= Q(0)$, $0\le t \le T- 1$. Similarly, $y_t^{(k+1,2j)} =  Q(0)$, $0\le t \le T- 1$. Therefore,
\begin{equation}
\label{eq:tystukp1tm1d}
y_t^{(k+1,2j-1)} = y_t^{(k+1,2j)}, \quad t =  0,1,\dots,T-1.
\end{equation}
At $t = T$, $y_T^{(k+1,2j-1)} = Q(CA^T s_{(k+1,2j-1)} )$, and $y_T^{(k+1,2j)} = Q(CA^T s_{(k+1,2j)})$. Recall \eqref{eq:skjexplic}, if $ \alpha_1^{(k+1,2j-1)} = \alpha_1^{(k+1,2j)} = 0$, then $CA^T s_{(k+1,2j-1)} < \frac{1}{\lambda^T-1}\beta$.
Choose $T\in\mathbb{Z}_+$ such that 
\begin{equation}
\label{eq:ncc2cT4}
\lambda^T - 1 > 1.
\end{equation}
Then  $y_T^{(k+1,2j-1)} = Q(0)$. Similarly, we can show that $y_T^{(k+1,2j)} = Q(0)$. If $ \alpha_1^{(k+1,2j-1)} = \alpha_1^{(k+1,2j)} = 1$, then $\beta \le CA^T s_{(k+1,2j-1)} <  \beta \frac{1}{1-\lambda^{-T}}$. Recall \eqref{eq:ncc2cT2}, we see that $CA^T s_{(k+1,2j-1)} \in [\beta, \beta + \delta_1)$, and therefore $y_T^{(k+1,2j-1)}  = Q(\beta)$. Similarly,  $y_T^{(k+1,2j)} = Q(\beta)$. We summarize the preceding as
\begin{equation}
\label{eq:tystukp1tm1d2}
y_t^{(k+1,2j-1)} = y_t^{(k+1,2j)}, \quad t =  T.
\end{equation}

At $t = T+1$, recall \eqref{eq:utkjexpltp1}, $u_{T+1}^{(k+1,2j-1)} = \alpha_1^{(k+1,2j-1)} u^*, u_{T+1}^{(k+1,2j)} = \alpha_1^{(k+1,2j)} u^*.$ Recall \eqref{eq:alphakjprop}, $ \alpha_1^{(k+1,2j-1)} = \alpha_1^{(k+1,2j)} = \alpha_1^{(k,j)}$. If $ \alpha_1^{(k,j)} = 0$,  recall \eqref{eq:skjexplic}, then $CA^{T+1} s_{(k+1,2j-1)} <  \frac{\lambda}{\lambda^T-1}\beta$.
Choose $T\in\mathbb{Z}_+$ such that 
\begin{equation}
\label{eq:ncc2cT5}
\lambda^T - 1 > \lambda.
\end{equation}
Then   $y_{T+1}^{(k+1,2j-1)} = y_{T+1}^{(k+1,2j)} = Q(0)$. If $ \alpha_1^{(k,j)} = 1$, and therefore $u_{T+1}^{(k+1,2j-1)} = u^*$, then $\lambda \beta + Du^* \le CA^{T+1} s_{(k+1,2j-1)} + Du^* < \lambda \beta\frac{1}{1-\lambda^{-T}} + Du^*$. Recall \eqref{eq:ncc2cT3}, we see that $y_{T+1}^{(k+1,2j-1)} = y_{T+1}^{(k+1,2j)} = Q(\lambda\beta + Du^*)$. We conclude that 
\begin{equation}
\label{eq:tystukp1tm1d3}
y_t^{(k+1,2j-1)} = y_t^{(k+1,2j)}, \quad t =  T+1.
\end{equation}

At $t = T+2$, for any $1\le j\le 2^k$, recall \eqref{eq:skjexplic}, \eqref{eq:utkjexpltp1}, we see that 
$
x_{T+2}^{(k+1,2j-1)} =A^{2} (\alpha_2^{(k+1,2j-1)}    + \cdots + \alpha_{k+1}^{(k+1,2j-1)}  q^{k-1})  s_o + \alpha_1^{(k+1,2j-1)} (\lambda^{T+2}  s_o +B u^*).
$
Note that 
$
\lambda^{T+2}  s_o +B u^* = A^2 x^* + Bu^* = 0
$, therefore 
\begin{equation}
\label{eq:xtp2kp1}
x_{T+2}^{(k+1,2j-1)}  =  A^{2} (\alpha_2^{(k+1,2j-1)}    + \cdots + \alpha_{k+1}^{(k+1,2j-1)}  q^{k-1})  s_o.
\end{equation}
Consider $x_{2}^{(k,h(k+1,2j-1))}$, recall  \eqref{eq:skjexplic}, \eqref{eq:utkjexplaz1},
\begin{equation}
\label{eq:xt2kp12j}
\begin{aligned}
&x_{2}^{(k,h(k+1,2j-1))} = \\
&\quad A^{2} (\alpha_1^{(k,h(k+1,2j-1))}    + \cdots + \alpha_k^{(k,h(k+1,2j-1))}  q^{k-1})  s_o.
\end{aligned}
\end{equation}
In the following, we show that $x_{T+2}^{(k+1,2j-1)} = x_{2}^{(k,h(k+1,2j-1))}$.

Recall \eqref{eq:alphakj}, \eqref{eq:funchkj}, note that $(j-1) \hspace{-4pt}\mod (2^{k-l+1}) = h(k-l+2,j)-1$, we see that 
\begin{equation}
\label{eq:alphalkjalf}
\alpha_l^{(k,j)} = \left\{\begin{array}{cl}
0, & \textrm{if} \ 0 \le h(k-l+2,j)-1 < \frac{2^{k-l+1}}{2}, \\
1, & \textrm{if} \ \frac{2^{k-l+1}}{2} \le h(k-l+2,j)-1 < 2^{k-l+1}, \\
\end{array}\right.
\end{equation}
for any $k \in \mathbb{Z}_+$, $j \in \{1,\dots, 2^k\}$, and $l \in \{1,\dots, k\}$.

Compare \eqref{eq:xtp2kp1} and \eqref{eq:xt2kp12j}, and recall \eqref{eq:funchkj}, \eqref{eq:funchkjprop}, we observe that for any $2 \le  l \le k+1$, 
\begin{equation}
\label{eq:hkjhhkjprop}
\begin{aligned}
h((k+1)-l+2,(2j-1)) &= 2 h(k-l+2,j) - 1,\\
h(k-l+3,h(k+1,2j-1)) & = 2 h(k-l+2, h(k,j) )-1,\\
h(k-l+2,j) &= h(k-l+2, h(k,j) ).
\end{aligned}
 \end{equation}
Based on the above, we can show that for any $2 \le  l \le k+1$,
$
\alpha_l^{(k+1,2j-1)} = \alpha_{l-1}^{(k,h(k+1,2j-1))}.
$
Recall \eqref{eq:xtp2kp1}, \eqref{eq:xt2kp12j},  for any $1\le j\le 2^k$,
$
x_{T+2}^{(k+1,2j-1)} = x_{2}^{(k,h(k+1,2j-1))}. 
$
Similarly, we can show that 
$
x_{T+2}^{(k+1,2j)} = x_{2}^{(k,h(k+1,2j))}. 
$
Therefore, for any $1\le j\le 2^{k+1}$,
\begin{equation}
\label{eq:xstp2utp2}
x_{T+2}^{(k+1,j)} = x_{2}^{(k,h(k+1,j))}. 
\end{equation}
Recall \eqref{eq:utkjexplss}, $u_t^{(k+1,j)} = u_{t-T}^{(k,h(k+1,j))},  T+2 \le t \le (k+1)T$, or equivalently,
\begin{equation}
\label{eq:ustptutp2}
u_{T+t}^{(k+1,j)} = u_{t}^{(k,h(k+1,j))}, \ 2 \le t \le kT.
\end{equation}
By \eqref{eq:xstp2utp2}, \eqref{eq:ustptutp2}, and the time-invariance of system \eqref{eq:Sofa-LTI}, we see that for any $1\le j\le 2^{k+1}$, 
$y_{t+T}^{(k+1,j)} = y_{t}^{(k,h(k+1,j))}, \ 2 \le t \le kT.$ 
Recall \eqref{eq:funchkjprop}, for any $1\le j\le 2^{k}$,
\begin{equation}
\label{eq:tystptukp1}
\left\{ \begin{array}{l}
y_{t+T}^{(k+1,2j-1)} = y_{t}^{(k,2h(k,j)-1)}, \\
 y_{t+T}^{(k+1,2j)} = y_{t}^{(k,2h(k,j))}, 
 \end{array} \right. \ 2 \le t \le kT.
\end{equation} 

By assumption, item \ref{item:IOfamydif1}) is satisfied for $k$. Recal \eqref{eq:IOfamydif1}, note that $1 \le h(k,j) \le 2^{k-1}$, we see that 
$
 y_{t}^{(k,2h(k,j)-1)} = y_{t}^{(k,2h(k,j))}, 2 \le t \le kT -1; 
 y_{t}^{(k,2h(k,j)-1)} \neq y_{t}^{(k,2h(k,j))}, t = kT.
$ 
Recall \eqref{eq:tystptukp1}, we see that for any $1\le j\le 2^{k}$,
\begin{equation}
\label{eq:tystptukp12jm1}
\begin{aligned}
y_{t}^{(k+1,2j-1)} &= y_{t}^{(k+1,2j)}, \quad  T + 2 \le t \le (k+1)T - 1, \\
y_{t}^{(k+1,2j-1)} &\neq y_{t}^{(k+1,2j)},  \quad t = (k+1)T.
 \end{aligned}
\end{equation} 
Recall \eqref{eq:famiiT}, \eqref{eq:ustukp1d2j}, \eqref{eq:tystukp1tm1d}, \eqref{eq:tystukp1tm1d2}, \eqref{eq:tystukp1tm1d3}, \eqref{eq:tystptukp12jm1}, we see that  item \ref{item:IOfamydif1}) is satisfied for $k+1$.

For item \ref{item:IOfamconn}), recall \eqref{eq:famiiinc}, we see that $u_t^{(k+1,2j-1)} = u_t^{(k,j)}, 0 \le t \le k T$, for any $1\le j \le 2^{k}$. Recall \eqref{eq:faminitsinc}, $s_{(k+1, 2j-1)}  = s_{(k, j)}$. Therefore $y_{t}^{(k+1,2j-1)} = y_{t}^{(k,j)}, 0 \le t \le k T$, for any $1\le j \le 2^{k}$. Since  item \ref{item:IOfamydif1}) is satisfied for $k+1$, $y_{t}^{(k+1,2j-1)} = y_{t}^{(k+1,2j)},  0 \le t \le k T$, and therefore item \ref{item:IOfamconn}) is satisfied for $k+1$. 

Choose $T \ge 2$ such that \eqref{eq:ncc2cT1},  \eqref{eq:ncc2cT2}, \eqref{eq:ncc2cT3}, \eqref{eq:ncc2cT4}, and \eqref{eq:ncc2cT5} are satisfied, by induction, we conclude that $\Psi$ \eqref{eq:famio} satisfies items \ref{item:IOfamydif1}) and \ref{item:IOfamconn}) in Theorem \ref{prop:NC-C2}.

\end{proof}

Next, we show that $\Psi$ \eqref{eq:famio} satisfies item \ref{item:IOcomplete}) in Theorem \ref{prop:NC-C2}.
\begin{lemma}
\label{lem:Psisfyit3}
Given a system \eqref{eq:Sofa-LTI} that satisfies the hypotheses in Theorem \ref{prop:NC-C2-C3}, 
let $T \ge 2$ be such that  \eqref{eq:ncc2cT1},  \eqref{eq:ncc2cT2}, \eqref{eq:ncc2cT3}, \eqref{eq:ncc2cT4}, and \eqref{eq:ncc2cT5} are satisfied. Then $\Psi$ defined   in \eqref{eq:famio} satisfies  item \ref{item:IOcomplete}) in Theorem \ref{prop:NC-C2}. 
\end{lemma}

\begin{proof}
Given a sequence $\{j(k)\}_{k=1}^\infty$ that satisfies 
$j(k) \in \{1,\dots, 2^k\} \ \textrm{and} \ j(k+1) \in \{2j(k)-1, 2j(k)\},  \forall \ k \in \mathbb{Z}_+ 
$, we observe that $\lim_{k \to \infty}s_{(k,j(k))}$ exists. To see this,  recall \eqref{eq:skjexplic}, \eqref{eq:alphakjprop}, for any $k \in \mathbb{Z}_+$, we can show that 
$
s_{(k+1,j(k+1))} - s_{(k,j(k))} =\alpha_{k+1}^{(k+1,j(k+1))}  q^{k}  s_o.
$
Consequently, $\{s_{(k,j(k))}\}_{k=1}^\infty$ is a Cauchy sequence in $\mathbb{R}^n$. Since $\mathbb{R}^n$ is complete, $\{s_{(k,j(k))}\}_{k=1}^\infty$ converges.

Given $\{j(k)\}_{k=1}^\infty$, define an initial state $s$ of system \eqref{eq:Sofa-LTI} as
\begin{equation}
\label{eq:showcompleteinit}
s = \lim_{k \to \infty}s_{(k,j(k))}.
\end{equation}
Recall \eqref{eq:skjexplic} and $s_o \in span\{v\}$, we can show that 
\begin{equation}
\label{eq:limitspanclosed}
s \in span\{v\}.
\end{equation}

Define an input sequence ${\bf u}$ as
\begin{equation}
\label{eq:showcompleteinput}
\begin{aligned}
u_t & = u_t^{(1,j(1))}, \quad 0 \le t  \le T,\\
u_t & = u_t^{(k,j(k))},  \quad  (k-1)T< t  \le kT, \ \forall \ k \ge 2.
\end{aligned}
\end{equation}

In the following, let ${\bf {y}} = \{y_t\}_{t = 0}^\infty$ be the output , $\{x_t\}_{t=0}^\infty$ be the state of system \eqref{eq:Sofa-LTI} when its initial state is $x_0 = s$ \eqref{eq:showcompleteinit} and its input is ${\bf u}$ \eqref{eq:showcompleteinput}. 
For this ${\bf {y}}$, we observe that for any $k \in \mathbb{Z}_+$,
\begin{equation}
\label{eq:showcompletetildey}
y_t = y_t^{(k,j(k))}, \quad 0 \le t \le kT.
\end{equation}
We use induction to show this observation. 

For $k=1$, by the  derivation of Lemma \ref{lem:Psisfyit3}, we have $y_t^{(1,1)} = Q(0)$, $0 \le t \le T$, and  $y_t^{(1,2)} = Q(0)$, $0 \le t \le T-1$, $y_T^{(1,2)} = Q(\beta)$. If $j(1) = 1$,  recall \eqref{eq:skjexplic}, \eqref{eq:showcompleteinit}, we see that $Cs \le \frac{\beta}{\lambda^T} \frac{q}{1-q} $. 
Recall \eqref{eq:ncc2cT4} \eqref{eq:limitspanclosed}, \eqref{eq:showcompleteinput}, we see that  $y_t = Q(0) = y_t^{(1,j(1))}$ for $0 \le t \le T$. Similarly, 
if $j(1) = 2$, then $Cs \le \frac{\beta}{\lambda^T} \frac{1}{1-q}$.  Recall \eqref{eq:ncc2cT1}, we see that $y_t = Q(0) = y_t^{(1,j(1))}$ for $0 \le t \le T-1$. At $t = T$, recall \eqref{eq:skjexplic}, we see that $\beta \le Cx_T \le \beta \frac{1}{1-\lambda^{-T}}$. Recall \eqref{eq:ncc2cT2},  $y_T = Q(\beta) = y_T^{(1,j(1))}$, and therefore \eqref{eq:showcompletetildey} holds for $k=1$.

Assume \eqref{eq:showcompletetildey} holds for some $k\ge 1$, since $\Psi$ \eqref{eq:famio} satisfies items \ref{item:IOfamydif1}) and \ref{item:IOfamconn}) in Theorem \ref{prop:NC-C2}, recall \eqref{eq:IOfamydif1}, \eqref{eq:IOfamyuinc}, we see that  
$
y_t = y_t^{(k,j(k))} = y_t^{(k+1,j(k+1))},  0 \le t \le kT.
$
Consequently, to show \eqref{eq:showcompletetildey} holds for  $k+1$, we only need to consider $kT+1 \le t \le (k+1)T$.

We first make some observations about the quantities  $x_{kT}^{(k+1,j(k+1))}$ and $x_{kT}$. 
Recall  \eqref{eq:famiiinc}, \eqref{eq:utkjexpl},  \eqref{eq:hkjhhkjprop},   we see  that for all $k \ge 2$, $1 \le j \le 2^k$, $\mathcal{I}$ \eqref{eq:famii} satisfies:
\begin{equation}
\label{eq:utkjshowcompprop}
u_t^{(k,j)} = \left\{\begin{array}{ll}
\alpha_1^{(k-(l-1), h(k-(l-2),j))} u^*, & t \in \{l T + 1\}_{l=1}^{k-1}, \\
0, &\textrm{otherwise}. 
\end{array}\right.
\end{equation}
Next, we consider $x_{kT}^{(k+1,j(k+1))}$. Recall \eqref{eq:funchkj} \eqref{eq:alphalkjalf}, we see that  
\begin{equation}
\label{eq:alphalto1}
\alpha_{l}^{(k+1,j(k+1))} = \alpha_1^{(k-l+2, h(k-l+3,j(k+1)))}.
\end{equation}
Recall \eqref{eq:skjexplic}, \eqref{eq:alphakj},  \eqref{eq:utkjshowcompprop}, we can show that $
x_{kT}^{(k+1,j(k+1))} =   \sum_{l=1}^{k-1} A^{(k -1 - l)T} \alpha_{l}^{(k+1,j(k+1))}( A^{2T}  s_o +  A^{T-2} B u^*) +
  A^T \alpha_k^{(k+1,j(k+1))} s_o + \alpha_{k+1}^{(k+1,j(k+1))} s_o.
$
Note that $A^{2T}  s_o +  A^{T-2} B u^* = 0$, we have
\begin{equation}
\label{eq:xkTkp1jkp1}
x_{kT}^{(k+1,j(k+1))} =   A^T \alpha_k^{(k+1,j(k+1))} s_o + \alpha_{k+1}^{(k+1,j(k+1))} s_o.
\end{equation}

Next, consider $x_{kT}$, which is the system state corresponding with $x_0 = s$ \eqref{eq:showcompleteinit} and ${\bf u}$ \eqref{eq:showcompleteinput} at $t = kT$ . Recall \eqref{eq:IOfamyTeqioeq}, \eqref{eq:IOfamyuinc},  \eqref{eq:showcompleteinput}, we see that $u_t  = u_t^{(z,j(z))},  0 \le t  \le zT$, for any $z \in \mathbb{Z}_+$. Consequently, we see that 
\begin{equation}
\label{eq:xkTvsxkTstagekp1}
x_{kT}  = A^{kT} (s - s_{(k+1,j(k+1))}) + x_{kT}^{(k+1,j(k+1))}.
\end{equation}
Recall  \eqref{eq:skjexplic}, \eqref{eq:showcompleteinit}, we see that  
\begin{equation}
\label{eq:skttildesktdif}
0 \le CAx_{kT} - CAx_{kT}^{(k+1,j(k+1))} \le \lambda^{-T+1} \beta \frac{q}{1-q}. 
\end{equation}

We are now ready to show \eqref{eq:showcompletetildey} holds for  $k+1$ when $kT+1 \le t \le (k+1)T$. 

At $t  = kT+1$, recall \eqref{eq:utkjshowcompprop}, \eqref{eq:alphalto1}, $u_{kT+1} = u_{kT+1}^{(k+1,j(k+1))} = \alpha_{k}^{(k+1,j(k+1))} u^*$. Recall \eqref{eq:xkTkp1jkp1},  $Cx_{kT+1}^{(k+1,j(k+1))} + Du_{kT+1}^{(k+1,j(k+1))}  = \alpha_{k}^{(k+1,j(k+1))} (CA^{T+1} s_o + Du^*) + CA\alpha_{k+1}^{(k+1,j(k+1))} s_o$.  If $\alpha_{k}^{(k+1,j(k+1))} = 0$, we have $y_{kT+1}^{(k+1,j(k+1))} = Q(0)$. Recall \eqref{eq:xkTkp1jkp1}, \eqref{eq:skttildesktdif},  $Cx_{kT+1} + Du_{kT+1} \le \lambda \frac{1}{\lambda^T-1} \beta$. Recall \eqref{eq:ncc2cT5}, $y_{kT+1} = Q(0) = y_{kT+1}^{(k+1,j(k+1))}$. If $\alpha_{k}^{(k+1,j(k+1))} = 1$,  $Cx_{kT+1}^{(k+1,j(k+1))} + Du_{kT+1}^{(k+1,j(k+1))} \in [\lambda\beta + Du^*, \lambda\beta + Du^* + \lambda \beta \frac{1}{\lambda^T-1})$, and $Cx_{kT+1} + Du_{kT+1} \in [\lambda\beta + Du^*, \lambda\beta + Du^* + \lambda \beta \frac{1}{\lambda^T-1}]$. Recall \eqref{eq:ncc2cT3}, $\lambda \beta \frac{1}{\lambda^T-1} < \delta_2$, and therefore $y_{kT+1} =  y_{kT+1}^{(k+1,j(k+1))} = Q(\lambda\beta + Du^*)$. We conclude that $y_{t} =  y_{t}^{(k+1,j(k+1))}$, $t = kT+1$.

By \eqref{eq:utkjshowcompprop}, \eqref{eq:alphalto1}, \eqref{eq:xkTkp1jkp1},  we have 
$
x_{kT+2}^{(k+1, j(k+1))} = A^2 \alpha_{k+1}^{(k+1,j(k+1))} s_o.
$
Recall \eqref{eq:utkjshowcompprop}, we can show that $y_{t}^{(k+1,j(k+1))} = Q(0), kT+2 \le t \le (k+1)T - 1$. Recall \eqref{eq:showcompleteinput},  \eqref{eq:xkTvsxkTstagekp1}, we can show that 
$
Cx_t + Du_t \le  \lambda^{-1} \beta\frac{1}{1-q}.
$
Recall \eqref{eq:ncc2cT1}, $\lambda^{-1} \beta\frac{1}{1-q} < \beta$, therefore $y_t = Q(0) = y_{t}^{(k+1,j(k+1))}, kT+2 \le t \le (k+1)T - 1$.

At $t = (k+1)T$, recall \eqref{eq:utkjshowcompprop}, we see that $x_{kT+T}^{(k+1, j(k+1))} = A^T \alpha_{k+1}^{(k+1,j(k+1))} s_o$. Recall \eqref{eq:xkTvsxkTstagekp1}, we have $CA^T \alpha_{k+1}^{(k+1,j(k+1))} s_o \le Cx_{(k+1)T} \le CA^T \alpha_{k+1}^{(k+1,j(k+1))} s_o + \frac{q}{1-q}\beta$.
If $\alpha_{k+1}^{(k+1,j(k+1))} = 0$, recall \eqref{eq:ncc2cT4}, we have $y_{(k+1)T}  = y_{(k+1)T}^{(k+1,j(k+1))} = Q(0)$. If $\alpha_{k+1}^{(k+1,j(k+1))} = 1$, recall  \eqref{eq:ncc2cT2}, we have $y_{(k+1)T}  = y_{(k+1)T}^{(k+1,j(k+1))} = Q(\beta)$. We conclude that $y_{t} =  y_{t}^{(k+1,j(k+1))}$, $t = (k+1)T$.

So far, we have shown that \eqref{eq:showcompletetildey} holds for $k+1$. By induction, we conclude that \eqref{eq:showcompletetildey} holds for all $k \in \mathbb{Z}_+$.

Finally, we show that the pair $({\bf u, y})$, which corresponds with the initial state $s$ \eqref{eq:showcompleteinit} and the input \eqref{eq:showcompleteinput}, satisfies \eqref{eq:IOcompletedef}. Recall \eqref{eq:famiiT}, $T_{(k,j(k))} = k\cdot T,$ for all $k \in \mathbb{Z}_+$. For $k=1$, by \eqref{eq:showcompleteinput}, $u_t  = u_t^{(1,j(1))},  0 \le t  \le T$. And by \eqref{eq:showcompletetildey}, $y_t  = y_t^{(1,j(1))},  0 \le t  \le T$. For any $k \ge 2$, by \eqref{eq:showcompleteinput}, $u_t  = u_t^{(k,j(k))},   (k-1)T< t  \le kT$. By \eqref{eq:showcompletetildey}, $y_t  = y_t^{(k,j(k))},  0 \le t  \le kT$, and consequently $y_t  = y_t^{(k,j(k))},   (k-1)T< t \le kT$. Therefore $({\bf u, y})$ satisfies \eqref{eq:IOcompletedef}. Note that $({\bf u, y}) \in P$ \eqref {eq:Sofa-LTI}, we conclude that $\Psi$ \eqref{eq:famio} satisfies item \ref{item:IOcomplete}) in Theorem \ref{prop:NC-C2}.

\end{proof}

We are now ready to show Theorem \ref{prop:NC-C2-C3}.

\begin{proof} (Theorem \ref{prop:NC-C2-C3})

Let $T \ge 2$ be such that  \eqref{eq:ncc2cT1},  \eqref{eq:ncc2cT2}, \eqref{eq:ncc2cT3}, \eqref{eq:ncc2cT4}, and \eqref{eq:ncc2cT5} are satisfied,  by Lemma \ref{lem:Psisfyit12}  and Lemma \ref{lem:Psisfyit3}, 
 $\Psi$ defined in \eqref{eq:famio} satisfies the hypotheses in Theorem \ref{prop:NC-C2}, consequently system \eqref{eq:Sofa-LTI} is not weakly output observable.

Next, we show that system \eqref{eq:Sofa-LTI} is not asymptotically output observable.   For any observer $\hat{S}$,  there is $({\bf u, {y}})$, which corresponds with the initial state $s$ \eqref{eq:showcompleteinit} and the input ${\bf u}$ \eqref{eq:showcompleteinput}, such that for all $ k \in \mathbb{Z}_+$, 
$
\hat{y}_t \neq y_t,  t = kT.
$
Let $\delta = \min\{\|y_1 - y_2\| : y_1 \neq y_2, y_1, y_2 \in \mathcal{Y}\}$, and define $\gamma = \frac{\delta}{2\|u^*\|} > 0$. For any $N \in \mathbb{Z}_+$ and $N \ge 2$,
$
\sum_{t=T+1}^{NT} \|y_t - \hat{y}_t\| - \gamma \|u_t\| = \sum_{k=2}^N(\sum_{t = (k-1)T+1}^{kT} (\|y_t - \hat{y}_t\| - \gamma \|u_t\| )).
$
Recall \eqref{eq:showcompleteinput}, \eqref{eq:utkjshowcompprop}, $\sum_{t = (k-1)T+1}^{kT} \|u_t\| = \|u^*\|$. Since $\hat{y}_t \neq y_t, t = kT$, $(\sum_{t = (k-1)T+1}^{kT} \|y_t - \hat{y}_t\|) \ge \|y_{kT}- \hat{y}_{kT}\| \ge \delta.$ Therefore
$
\sum_{t = (k-1)T+1}^{kT} (\|y_t - \hat{y}_t\| - \gamma \|u_t\| ) \ge \delta - \frac{\delta}{2\|u^*\|} \|u^*\| = \frac{\delta}{2}.
$
Consequently, 
$
\sum_{t=T+1}^{NT} \|y_t - \hat{y}_t\| - \gamma \|u_t\| \ge (N-1) \frac{\delta}{2}, \ \forall \ N \ge 2.
$
Therefore, $\sup_{N \ge 2} \sum_{t=T+1}^{NT} \|y_t - \hat{y}_t\| - \gamma \|u_t\| = \infty$. Since $\mathcal{U}, \mathcal{Y}$ are finite sets, $ \sum_{t=0}^{T} \|y_t - \hat{y}_t\| - \gamma \|u_t\|$ is finite, and consequently  $\sup_{T \ge 0} \sum_{t=0}^{T} \|y_t - \hat{y}_t\| - \gamma \|u_t\| = \infty$. By definition,  $\gamma = \frac{\delta}{2\|u^*\|} > 0$ is not an observation gain bound of system \eqref{eq:Sofa-LTI}. Recall Definition \ref{def:obsr-gain}, we see that for any $\gamma' \in \mathbb{R}_{\ge 0}$, if $\gamma' < \gamma$, then $\gamma'$ is not an observation gain bound. Recall \eqref{eq:o-gain}, we see that the  $\mathcal{O}$-gain $\gamma^*$ of system \eqref{eq:Sofa-LTI} satisfies $\gamma^* \ge \gamma > 0$. Recall Definition \ref{def:class-C}, system \eqref{eq:Sofa-LTI} is not asymptotically output observable.
\end{proof}

\section{Illustrative Examples}
\label{sec:Eg}

We begin with an example that highlights a subtle distinction  between the concept of finite memory output observability and  observability of LTI systems in the traditional sense.
In particular, in the traditional LTI setting,
the effects of the initial state of a stable LTI system will die down eventually, 
and the question of observability is only interesting for unstable systems.
This is not the case for our class of systems of interest, as a system with stable internal dynamics may retain a recollection of its initial state ``forever" in the output.

\begin{eg}
\label{Eg:Example1}
Consider system \eqref{eq:Sofa-LTI} with parameters: $A = 0.5, B = 1, C = 1, D = 1$, $\mathcal{U} = \{0, \pm 1\}$.
The quantizer $Q$ with $p = 1$ is described: 
\begin{equation}
\label{eq:1d-quantizer}
Q(y)=\left\{ \begin{array}{ll}
 i, \quad \textrm{$y \in [i-0.5, i+0.5)$ for $i \in \mathbb{Z}$ and $|i| \le R$}\\
 \lfloor R \rfloor,  \qquad \quad \! y \ge  \lfloor R \rfloor + 0.5\\
 -\lfloor R \rfloor,  \quad \quad y < - \lfloor R \rfloor - 0.5
 \end{array} \right.
\end{equation}
where {$R \in \mathbb{R}_+$} is a parameter. Let $R = 1$, and (consequently) $\mathcal{Y} = \{0, \pm 1\}$.  

Given $u_t$ and $y_t$ for all $0 \le t \le T$ for some $T \in \mathbb{Z}_+$, and an arbitrary $u_{T+1} \in \mathcal{U}$, we claim that $y_{T+1}$ cannot be uniquely determined. 

To see this claim, assume the contrary and let $u_t = 0$ for $0 \le t \le T-2$, $u_{T-1} = 1$ and $u_{T} = u_{T+1} = 0$. For two distinct initial states $x_0^1 = 0.1$ and $x_0^2 = -0.1$, we use $y_t^1$ and $y_t^2$ to denote the quantized outputs respectively. Then $y_t^1 = y_t^2 = 0$ for $0 \le t \le T-2$, and $y_t^1 = y_t^2 = 1$ for $T-1 \le t \le T$. By assumption we can uniquely determine $y_{T+1}$, which contradicts with $y_{T+1}^1 = 1$ and $y_{T+1}^2 = 0$. \hfill \qedsymbol
\end{eg}

As shown in Example \ref{Eg:Example1}, 
the initial state of system \eqref{eq:Sofa-LTI} impacts the quantized output at arbitrarily large times, 
even though the underlying LTI system is stable. 
Consequently, the question of finite memory output observability remains relevant even when the internal dynamics are stable.

The next two examples are instances of system \eqref{eq:Sofa-LTI} that are finite memory output observable. Nonetheless, the two underlying linear dynamics are (LTI) observable in one example, but not the other. This highlights the fact that there is not direct link between finite memory output observability and observability of the underlying LTI dynamics.
\begin{eg}
\label{eg:e1}
Consider a system \eqref{eq:Sofa-LTI} whose LTI parameters are:
$
A = \left[ \begin{array}{cc}
0.25 & -0.05 \\
0 & 0.2 \end{array} \right],  
B = \left[ \begin{array}{c}
2 \\
1 \end{array} \right], 
C = \left[ \begin{array}{cc}
0.5 & 0 \end{array} \right], 
D = 1. 
$ 
Note that this is a minimal representation.  $\mathcal{U} = \{ 0, 1, -1\}$ and the quantizer $Q(\cdot)$ is defined in  \eqref{eq:1d-quantizer} with $R = 5$. Next we  show this system satisfies the hypotheses in Theorem \ref{prop:SC-C1-G}.
\end{eg}

We assume that the initial state $x(0)$ of the LTI system is bounded, particularly:
$
\|x_0\|_{\infty} < b
$
for some $b \in \mathbb{R}_+$.

First we find the distance $d(\mathcal{A}, \mathcal{B})$ between the two sets $\mathcal{A}$ and $\mathcal{B}$ defined in \eqref{eq:set-A} and \eqref{eq:set-B}. 
Since $A$ is diagonalizable, 
we have
$
A^n = \left[ \begin{array}{cc}
(1/4)^n & (1/5)^n - (1/4)^n \\
0 & (1/5)^n \end{array} \right], 
$
and 
$
CA^{n}B = 1/2 ((1/5)^n + (1/4)^n), \forall n \in \mathbb{N}.
$
Consequently wee can show that
$
d(\mathcal{A}, \mathcal{B}) = 5/24.
$
This means that the forced response of the underlying LTI system is at least $5/24$ away from any discontinuous point of the quantizer.

Based on the derivation of  Theorem \ref{prop:SC-C1-G}, we construct an observer for this system. Note that $\|x_t\|_{\infty} \le \frac{10} {7} \max\{b, 2\}$,  choose 
$
T = \lceil{\log_4 \frac{72} {7} \max\{b, 2\}}\rceil + 1
$, and construct an observer $\hat{S}$ according to Definition \ref{def:FIIC} in Section \ref{sec:Deri-C1}. By the derivation of  Theorem \ref{prop:SC-C1-G}, we can show that the output $\hat{y}_t$ of $\hat{S}$ satisfies $\hat{y}_t = y_t, \forall \ t \ge T$.
 \hfill \qedsymbol

\begin{eg}
\label{eg:e2}
We present another second order system \eqref{eq:Sofa-LTI} that is also observable. The parameters of the LTI system in \eqref{eq:Sofa-LTI} are:
$
A = \left[ \begin{array}{cc}
2 & 2 \\
0 & 0 \end{array} \right] 
B = \left[ \begin{array}{c}
1 \\
1 \end{array} \right]
C = \left[ \begin{array}{cc}
0 & 1 \end{array} \right]
D = 1, 
$ 
$\mathcal{U} = \{ 0, 1, -1\}$.  The quantizer $Q(\cdot)$ is defined in  \eqref{eq:1d-quantizer}, and $R = 5$. 
\end{eg}

Clearly $CA = {\bf 0}$, so this system satisfies the condition in Theorem \ref{prop:SC-2-C1}. Notice that the solution of $\tilde{y}_t$ is:
$ \tilde{y}_t = u_{t-1} + u_t, \quad \forall \quad t \ge 2$.
It is thus straightforward to construct an observer that achieves $\gamma = 0$, described as follows:
\begin{align*}
q_{t+1} &=  u_t\\
\hat{y}_t &=  Q(q_t + u_t)
\end{align*}
where $t \in \mathbb{N}$, $q_t \in \mathcal{Q}$ and $\mathcal{Q} = \mathcal{U}$. $q_0$ can be arbitrary, say $0$. \hfill \qedsymbol

As discussed in Section \ref{sec:defn-obs}, the definition of finite memory output observability is stronger than that of weak output observability, and we provide a concrete example to show this point. In particular, our next example is an instance of system \eqref{eq:Sofa-LTI} that is weakly output observable but not finite memory output observable.

\begin{eg}
\label{eg:DFM-NZI}
Given system \eqref{eq:Sofa-LTI} with parameters $A = 0.5, B = C = 1, D = 0$, the input set is $\mathcal{U} = \{0, 1, -1\}$, and the initial state $x_0$ satisfies $|x_0| < 2$. The quantizer $Q$ is described by: $Q((-\infty, 0.5)) = 0$, and $Q([0.5, \infty)) = 1$.
\end{eg}

By Theorem \ref{prop:NC-C1}, the system in Example \ref{eg:DFM-NZI} is not finite memory output observable. However, we can design an observer $\hat{S}$ that achieves the observation gain bound $\gamma = 0$. This particular design keeps track of the last two steps of input as well as last three steps of nonzero input of system \eqref{eq:Sofa-LTI}. Therefore this system is weakly output observable. \hfill \qedsymbol

Lastly, we present a one-dimensional system \eqref{eq:Sofa-LTI} that is not  asymptotically output observable. 

\begin{eg}
Consider a system \eqref{eq:Sofa-LTI} with  parameters: $A = 2, B = 1, C = 1, D = 0$. The quantizer $Q$ is described by: $Q((-\infty, 0.5)) = 0, Q([0.5, \infty)) = 1$. The input set is $\mathcal{U} = \{0, 2, -2\}$. 

Let $x^* = 0.5$, $u^* = -2$, and note that the discontinuous point of $Q$ is $\beta = 0.5$, then $Cx^* = 1 \cdot 0.5 = \beta$, $A^2 x^* + Bu^* = -2 + 2^2 \cdot 0.5 = 0$, therefore the hypotheses in Theorem \ref{prop:NC-C2-C3} are satisfied, and consequently this system is not asymptotically output observable. \hfill \qedsymbol
\end{eg}

\section{Conclusions and Future work}
\label{Sec:Conclusions}
In this manuscript, we formulate the notion of observability of systems over finite alphabets in the sense of how well the output of the system could be estimated based on past input and output information. We characterize this proposed notion by deriving both necessary and sufficient conditions of observability in terms of system parameters. For system \eqref{eq:Sofa-LTI}, such conditions involve both the dynamics of the underlying LTI system and the discontinuous points of the quantizer. Regarding future directions, we are interested in investigating the case when $d(\mathcal{A}, \mathcal{B}) = 0$ and $\mathcal{A} \cap \mathcal{B} = \varnothing.$ We also want to pursue a sufficient condition of weak output observability that is weaker than the proposed sufficient conditions of finite memory output observability. 

\addtolength{\textheight}{-3cm}   

\color{black}

\bibliographystyle{ieeetr}
\bibliography{bib1.bib}

\begin{thebibliography}{10}

\bibitem{DT}
D.~C. Tarraf, A.~Megretski, and M.~A. Dahleh, ``A framework for robust
  stability of systems over finite alphabets,'' {\em IEEE Transactions on
  Automatic Control}, vol.~53, no.~5, pp.~1133--1146, 2008.

\bibitem{DT12}
D.~C. Tarraf, ``A control-oriented notion of finite state approximation,'' {\em
  IEEE Transactions on Automatic Control}, vol.~57, no.~12, pp.~3197--3202,
  2012.

\bibitem{DT11}
D.~C. Tarraf, ``Finite approximations of switched homogeneous systems for
  controller synthesis,'' {\em IEEE Transactions on Automatic Control},
  vol.~59, no.~5, pp.~1140--1145, 2011.

\bibitem{Hespanha}
J.~P. Hespanha, {\em Linear Systems Theory}.
\newblock Princeton University Press, 2009.

\bibitem{hermann1977nonlinear}
R.~Hermann and A.~J. Krener, ``Nonlinear controllability and observability,''
  {\em IEEE Transactions on Automatic Control}, vol.~22, no.~5, pp.~728--740,
  1977.

\bibitem{Balluchi}
A.~Balluchi, L.~Benvenuti, M.~D. Di~Benedetto, and A.~L.
  Sangiovanni-Vincentelli, ``Observability for hybrid systems,'' in {\em
  Proceedings of the 42nd IEEE Conference on Decision and Control}, (Maui, HI),
  pp.~1159--1164, 2003.

\bibitem{DL}
A.~Tanwani, H.~Shim, and D.~Liberzon, ``Observability for switched linear
  systems: Characterization and observer design,'' {\em IEEE Transactions on
  Automatic Control}, vol.~58, no.~4, pp.~891--904, 2013.

\bibitem{Xie}
G.~Xie and L.~Wang, ``Necessary and sufficient conditions for controllability
  and observability of switched impulsive control systems,'' {\em IEEE
  Transactions on Automatic Control}, vol.~49, no.~6, pp.~960--966, 2004.

\bibitem{DD}
D.~F. Delchamps, ``Extracting state information from a quantized output
  record,'' {\em Systems \& Control Letters}, vol.~13, pp.~365--372, 1989.

\bibitem{JS}
J.~Sur and B.~Paden, ``Observers for linear systems with quantized outputs,''
  in {\em Procedings of the American Control Conference}, (Albuquerque, NM),
  pp.~3012--3016, 1997.

\bibitem{Raisch}
J.~Raisch, ``Controllability and observability of simple hybrid control
  systems-{FDLTI} plants with symbolic measurements and quantized control
  inputs,'' in {\em International Conference on Control'94}, vol.~1, (Coventry,
  UK), pp.~595--600, 1994.

\bibitem{yuksel2007communication}
S.~Y{\"u}ksel and T.~Ba{\c{s}}ar, ``Communication constraints for decentralized
  stabilizability with time-invariant policies,'' {\em IEEE Transactions on
  Automatic Control}, vol.~52, no.~6, pp.~1060--1066, 2007.

\bibitem{yuksel2006minimum}
S.~Y{\"u}ksel and T.~Ba{\c{s}}ar, ``Minimum rate coding for lti systems over
  noiseless channels,'' {\em IEEE Transactions on Automatic Control}, vol.~51,
  no.~12, pp.~1878--1887, 2006.

\bibitem{DF14}
D.~Fan and D.~C. Tarraf, ``On finite memory observability of a class of systems
  over finite alphabets with linear dynamics,'' in {\em Proceedings of the 53rd
  IEEE Conference on Decision and Control}, (Los Angeles, CA), pp.~3884--3891,
  2014.

\bibitem{Murray06}
D.~Delvecchio, R.~M. Murray, and E.~Klavins, ``Discrete state estimators for
  systems on a lattice,'' {\em Automatica}, vol.~42, no.~2, pp.~271--285, 2006.

\bibitem{Topcu15}
R.~Ehlers and U.~Topcu, ``Estimator-based reactive synthesis under incomplete
  information,'' in {\em Proceedings of the 18th International Conference on
  Hybrid Systems: Computation and Control}, (Seattle, WA), pp.~249--258, 2015.

\bibitem{Ozay14}
O.~Mickelin, N.~Ozay, and R.~M. Murray, ``Synthesis of correct-by-construction
  control protocols for hybrid systems using partial state information,'' in
  {\em Proceedings of the American Control Conference, 2014}, (Portland, OR),
  pp.~2305--2311, 2014.

\bibitem{DT13}
D.~C. Tarraf, ``An input-output construction of finite state $\rho / \mu$
  approximations for control design,'' {\em IEEE Transactions on Automatic
  Control, Special Issue on Control of Cyber-Physical Systems}, vol.~59,
  no.~12, pp.~3164--3177, 2014.

\bibitem{Caro}
N.~L. Carothers, {\em Real Analysis}.
\newblock Cambridge University Press, 1999.

\bibitem{Khalil}
H.~K. Khalil, {\em Nonlinear Systems}.
\newblock Prentice Hall, 2002.

\bibitem{Horn}
R.~A. Horn and C.~R. Johnson, {\em Matrix Analysis}.
\newblock Cambridge University Press, 1990.

\bibitem{Bronson}
R.~Bronson, {\em Matrix Methods: An Introduction}.
\newblock Academic Press, 2014.

\bibitem{Nicholson}
W.~K. Nicholson, {\em Introduction to Abstract Algebra}.
\newblock Wiley, 2012.

\end{thebibliography}

\end{document}